\documentclass[11pt]{amsart}
\usepackage[all]{xy}
\usepackage{amsmath}
\usepackage{amsfonts}
\usepackage{amssymb}
\usepackage{latexsym}
\usepackage{graphicx}
\usepackage{color}
\usepackage{amscd}
\usepackage{epsfig}
\usepackage{cite}
\newtheorem{problem}{Problem}[section]
\newtheorem{definition}[problem]{Definition}
\newtheorem{lemma}[problem]{Lemma}
\newtheorem{theorem}[problem]{Theorem}

\newtheorem{corollary}[problem]{Corollary}

\newtheorem{conjecture}[problem]{Conjecture}

\title{The $p$-adic RH For Exponential Sums}
\author{Chunlei Liu and Chuanze Niu}
\address{Chunlei Liu, School of Mathematical Sciences, Shanghai Jiao Tong University, Shanghai 200240\\ email: clliu@sjtu.edu.cn}
\address{Chuanze Niu, School of Mathematical Sciences, Liaocheng University, Liaocheng 252059\\ email: niuchuanze@lcu.edu.cn}
\thanks{The second author is supported by Shandong Prov. NSF  No. ZR2019BA011  and by NNSFC No. 11401285}
\begin{document}
\maketitle
%==================================================================
{\bf Abstract. }
Let $(n,d;p)$ be a triple of natural numbers with $p$ being a prime number prime to $d$.
On the geometric side, that triple gives birth to the geometric invariant ${\rm HP}(n,d)$, which is the Hodge polygon of the simplex in ${\mathbb R}^n$ with vertices
  $$(0,\cdots,0),(d,0,\cdots,0), (0,d,0,\cdots,0),\cdots,(0,\cdots,0,d).$$
On the arithmetic side, the triple gives birth to the arithmetic invariant ${\rm GNP}(n,d;p)$, which is the generic Newton polygon of the space of polynomials in $n$ variables  of degree $d$ with coefficients in $\overline{{\mathbb F}}_p$. By a result of Sperber, if $p$ is sufficiently large, then
$${\rm GNP}(n,d;p)\geq{\rm HP}(n,d).$$
It is interesting to find a simple geometrically defined polygon that coincides with the generic Newton polygon ${\rm GNP}(n,d;p)$ for sufficiently large $p$.
If $p\equiv1 ({\rm mod}\ d)$, one needs no more efforts, because in this case it coincides with Hodge polygon.
If $n=1$, the problem was solved by  by Zhu [2003; 2004], and was simplified by
Blache and F\'{e}rard [2007].
If $n>1$ and $p\not\equiv1 ({\rm mod}\ d)$, the problem remains open. In present paper the authors have solved that problem when $n=2$, $d$ is even, and $p\equiv-1 ({\rm mod}\ d)$. In the general case, the newly defined geometric polygon is proved to lie between the generic Newton polygon and Hodge polygon.\\
%=========================================================
%==================================================================
%=================================================================

\tableofcontents
%=========================================================================================================
\section{Introduction}
\paragraph{}
Our starting ground object is the triple $(n,d;p)$ of natural numbers with $p$ being a prime number prime to $d$.
That ground triple, on the geometric side, gives rise to the simplex ${\rm Simp}(n,d)$, whose vertices are the points in ${\mathbb R}^n$:
  $$(0,\cdots,0),(d,0,\cdots,0), (0,d,0,\cdots,0),\cdots,(0,\cdots,0,d).$$
On the arithmetic side, the triple gives rise to
the space $\overline{{\mathbb F}}_p[x_1,\cdots,x_n]_{(d)}$, which is the set of degree-$d$ polynomials in $\overline{{\mathbb F}}_p[x_1,\cdots,x_n]$.

From the simplex ${\rm Simp}(n,d)$ mathematicians constructed the geometric invariant ${\rm HP}(n,d)$, which is called Hodge polygon of ${\rm Simp}(n,d)$.
And, from the space $\overline{{\mathbb F}}_p[x_1,x_2,\cdots,x_n]_{(d)}$  mathematicians constructed the arithmetic invariant ${\rm GNP}(n,d;p)$, which is called the generic Newton polygon of $\overline{{\mathbb F}}_p[x_1,x_2,\cdots,x_n]_{(d)}$.
By a result of Sperber, if $p$ is sufficiently large, then
$${\rm GNP}(n,d;p)\geq{\rm HP}(n,d).$$
It is interesting to find a simple geometrically defined polygon that coincides with the generic Newton polygon ${\rm GNP}(n,d;p)$ for sufficiently large $p$.
If $p\equiv1 ({\rm mod}\ d)$, one needs no more efforts, because in this case it coincides with Hodge polygon.
If $n=1$, the problem was solved by  by Zhu [2003; 2004], and was simplified by
Blache and F\'{e}rard [2007].
If $n>1$ and $p\not\equiv1 ({\rm mod}\ d)$, the problem remains open. In present paper the authors have solved that problem when $n=2$, $d$ is even, and $p\equiv-1 ({\rm mod}\ d)$. In the general case, the newly defined geometric polygon is proved to lie between the generic Newton polygon and Hodge polygon.
\subsection{Hodge numbers and Hodge polygon}
\begin{definition}[Hodge numbers] The Hodge numbers of ${\rm Simp}(n,d)$ are the numbers
                       $$h_j:=\#\{u\in(0,d)^n\cap{\mathbb Z}^n\mid \sum_{i=1}^nu_i=j\},\ j=0,1,\cdots,nd.$$

                       \end{definition}
Hodge numbers are symmetric in the sense that
$$h_{nd-j}=h_j,\ j=0,1,\cdots,nd.
$$
\begin{definition}[Hodge polygon]
The Hodge polygon ${\rm HP}(n,d)$ is the concave polygon on the right half plane with initial point $(0,0)$ whose slope distribution is listed in the following table

\begin{tabular}{|c|c|c|}
                                                      \hline

                                                      slope & multiplicity&variation\\
                                                      \hline
                                                      $\frac{j}{d}$ & $h_j$ &$j=0,1,\cdots,nd$\\
                                                      \hline
                                                    \end{tabular}\end{definition}
Hodge polygon is symmetric in the sense that
 the multiplicity of slope $\frac{j}{d}$ is equal to the multiplicity of slope $n-\frac{j}{d}$.
We can show that it has  end point
$(d-1)^n\cdot(1,\frac{n}{2})$.
\begin{definition}
An integer in the interval $[0,(d-1)^n]$ is called a Hodge vertex if it of the form
$\sum_{j=0}^ih_j$.\end{definition}
\subsection{Frobenius Numbers and Frobenius Polygon}
\paragraph{}
Frobenius numbers
$$h_{j,\epsilon},\ j=0,1,\cdots,nd,\ \epsilon=0,1$$
are to be defined in the next section.
They come from decompositions of Hodge numbers in the sense that
$$h_j=h_{j,0}+h_{j,1}.$$
They are trivial in trivial case in the sense that
$$h_{j,1}=0 \text{ if }(p-1)j\equiv0({\rm mod }\ d).$$
They are nonnegative in the sense that $$h_{j,\epsilon}\geq0.$$
Finally, they are symmetric in the sense that
$$h_{nd-j,1-\epsilon}=h_{j,\epsilon}\text{ if }(p-1)j\not\equiv0({\rm mod }\ d).$$
\begin{definition}[Frobenius slopes]
Frobenius slopes are the numbers:
$$\varpi_{j,\epsilon}:=\frac{1}{p-1}(\lceil\frac{(p-1)j}{d}\rceil-\epsilon),\ j=0,1,\cdots,nd,\ \epsilon=0,1.
$$\end{definition}
Frobenius slopes are  trivial in trivial case
in the sense that
$$\varpi_{j,0}=\frac{j}{d} \text{ if }(p-1)j\equiv0({\rm mod }\ d).$$
They are symmetric in the sense that
$$ \varpi_{nd-j,1-\epsilon}=n-\varpi_{j,\epsilon}\text{ if }(p-1)j\not\equiv0({\rm mod }\ d).$$
\begin{definition} The Frobenius polygon ${\rm FP}(n,d;p)$ is the concave polygon on the right half plane with initial point $(0,0)$ whose slope distribution is listed the the following table

\begin{tabular}{|c|c|c|}
                                                      \hline
                                                      slope & multiplicity&variation \\
                                                      \hline
                                                      $\varpi_{j,\epsilon}$ & $h_{j,\epsilon}$ &$j=0,1,\cdots,nd$ and $\epsilon=0,1$\\
                                                      \hline
                                                    \end{tabular}\end{definition}

Frobenius polygon has end point $(d-1)^n\cdot(1,\frac{n}{2})$.
It is  trivial in trivial case in the sense that
$${\rm FP}(n,d;p)={\rm HP}(n,d)\text{ if }p\equiv1({\rm mod }\ d).$$
It is asymptotically trivial
in the sense that
$$\lim_{p\rightarrow\infty}{\rm FP}(n,d;p)={\rm HP}(n,d).$$
It is symmetric in the sense that
multiplicity of slope $\varpi_{j,\epsilon}$ is equal to the multiplicity of slope $n-\varpi_{j,\epsilon}$.
In the third section we prove
that
$${\rm FP}(n,d;p)\geq{\rm HP}(n,d).$$
\begin{definition}
An integer in the interval $[0,(d-1)^n]$ is called a Frobenius vertex if it of the form
$\sum_{2j-\epsilon\leq 2i-\iota}h_{j,\epsilon}$.\end{definition}
\subsection{Discriminant and Smoothness}
\paragraph{}
To define the generic Newton polygon, we only need the family of polynomials whose leading form is smooth in the following sense.
\begin{definition}Let
$f(x)\in\overline{{\mathbb F}}_p[x_1,x_2,\cdots, x_n]_{(d)}$ be a homogenous polynomial.
If
 the system
 $$\left\{
     \begin{array}{ll}
       x_1\frac{\partial}{\partial x_1}f(x)=0& \hbox{} \\
       x_2\frac{\partial}{\partial x_2}f(x)=0& \hbox{}  \\
       \vdots & \hbox{} \\
       x_n\frac{\partial}{\partial x_n}f(x)=0& \hbox{}
     \end{array}
   \right.
 $$
has no nontrivial common zeros in $(\overline{{\mathbb F}}_p)^n$, then $f$ is said to be smooth.\end{definition}
\begin{lemma}[Gelfand et al.]There exists a nonzero polynomial
$${\rm Disc}(n,d)\in{\mathbb Z}[a_w\mid w\in{\mathbb N}^n,\ w_1+w_2+\cdots+w_n=d]$$
such that if $f(x)\in\overline{{\mathbb F}}_p[x_1,x_2,\cdots, x_n]_{(d)}$
is a homogenous polynomial which satisfies ${\rm Disc}(n,d)(f)\neq0$, then $f$ is smooth.\end{lemma}
\subsection{Exponential Sums and Generic Newton Polygon}
\begin{definition}Let
$f(x)\in\overline{{\mathbb F}}_p[x_1,x_2,\cdots, x_n]_{(d)}$, and $q$ a power of $p$ such that $\mathbb{F}_q$ contains the coefficients of $f$. The exponential sums associated to $(f,q)$
are the sums
$$S_{f,q}(k)=\sum\limits_{x\in({\mathbb F}_{q^k})^n}
\zeta_p^{{\rm
Tr}_{\mathbb{F}_{q^k}/\mathbb{F}_p}(f(x))},\ k=1,2,\cdots,$$
where $\zeta_p$ is a primitive $p$-th root of unity.\end{definition}
\begin{definition}Let
$f(x)\in\overline{{\mathbb F}}_p[x_1,x_2,\cdots, x_n]_{(d)}$, and $q$ a power of $p$ such that $\mathbb{F}_q$ contains the coefficients of $f$. The $L$-function
associated to $(f,q)$
is the function
$$L_{f,q}(t)=\exp\left(\sum\limits_{k=1}^{\infty}S_{f,q}(k)\frac{t^k}{k}\right).$$\end{definition}

\begin{theorem}[Deligne]Let
$f(x)\in\overline{{\mathbb F}}_p[x_1,x_2,\cdots, x_n]_{(d)}$, and $q$ a power of $p$ such that $\mathbb{F}_q$ contains the coefficients of $f$. If the leading form of $f$ is smooth, then
\begin{itemize}
  \item $L_{f,q}(t)^{(-1)^{n-1}}$ is a polynomial of degree $(d-1)^n$.
  \item $L_{f,q}(t)^{(-1)^{n-1}}$ is equal to $(q^{\frac{n}{2}}t)^{(d-1)^n}L_{-f,q}(\frac{1}{q^n t})^{(-1)^{n-1}}$ up to a root of unity.
   \item All reciprocal zeros of $L_{f,q}(t)^{(-1)^{n-1}}$ have absolute value $\sqrt{q^n}$.
\end{itemize}
\end{theorem}
\begin{definition}If
$f(x)\in\overline{{\mathbb F}}_p[x_1,x_2,\cdots, x_n]_{(d)}$ has smooth leading form, then the Newton polygon ${\rm NP}(f)$ is the concave polygon on the right half plane with initial point $(0,0)$ whose slope distribution is listed the the following table

\begin{tabular}{|c|c|c|}
                                                      \hline
                                                      slope & multiplicity&variation \\
                                                      \hline
                                                      ${\rm ord}_q(\alpha)$ & $m(\alpha^{-1})$ &$L_{f,q}(\alpha^{-1})^{(-1)^{n-1}}=0$\\
                                                      \hline
                                                    \end{tabular}

Here $q$ is any power of $p$ such that $\mathbb{F}_q$ contains the coefficients of $f$, and $m(\alpha^{-1})$ is the order of $\alpha^{-1}$ as a zero of $L_{f,q}(t)^{(-1)^{n-1}}$.\end{definition}
By Deligne's result, the Newton polygon ${\rm NP}(f)$ is independent of $q$,
has  end point
$(d-1)^n\cdot(1,\frac{n}{2})$, and is symmetric in the sense that
 the multiplicity of slope ${\rm ord}_q(\alpha)$ is equal to the multiplicity of slope $n-{\rm ord}_q(\alpha)$.
\begin{definition}The generic Newton polygon of $\overline{{\mathbb F}}_p[x_1,x_2,\cdots, x_n]_{(d)}$ is defined to be the polygon
$${\rm GNP}(n,d;p):=\inf_{{\rm Disc}(n,d)(f)\neq0}{\rm NP}(f),$$
where ${\rm Disc}(n,d)(f)$ is the discriminant of the leading form of $f$.\end{definition}
\begin{lemma}[Grothendieck's specialization lemma]There exists a nonzero polynomial
$${\rm Hass}(n,d;p)\in{\mathbb F}_p[a_w\mid w\in{\rm Simp}(n,d)\cap{\mathbb Z}^n]$$
such that if $f(x)\in\overline{{\mathbb F}}_p[x_1,x_2,\cdots, x_n]_{(d)}$ has smooth leading form and satisfies
${\rm Hass}(n,d;p)(f)\neq0$,
then
$${\rm NP}(f)={\rm GNP}(n,d;p).$$\end{lemma}
\subsection{Historical Results}
\paragraph{}
The result of Sperber [1986] implies the following.
\begin{theorem}[Sperber]\label{hodge-bound-1} If $p$ is sufficiently large, then
$${\rm GNP}(n,d;p)\geq{\rm HP}(n,d).
$$\end{theorem}
The well-known Stickelberger's theorem implies the following.
\begin{theorem}[Stickelberger]\label{Stickelberger}If $p\equiv1({\rm mod }\ d)$ is sufficiently large, then
$${\rm GNP}(n,d;p)= {\rm HP}(n,d)={\rm FP}(n,d;p).$$
\end{theorem}
Applications and  generalizations of the above theorem were intensively studied by mathematicians
such as Adolphson and  Sperber [1989], Wan [1993], Dwork [1973], Mazure [1972],
Deligne [1980], Illusie [1990], Katz [Mazure 1972], Koblitz [1975] and Miller [Koblitz 1975].

If $n=1$, a nontrivial result was proved by Zhu [2003; 2004]. The result was then simplified by
Blache and F\'{e}rard [2007].
\begin{theorem}[Zhu, Blache-F\'{e}rard]If $n=1$ and $p$ is sufficiently large, then
$${\rm GNP}(n,d;p)={\rm FP}(n,d;p).$$
\end{theorem}
\subsection{Main Results of Present Paper}
\paragraph{}
In present paper we shall prove the following theorems.
\begin{theorem}\label{main-theorem-lower-bound-1}If $p>2d$, then
$${\rm GNP}(n,d;p)\geq{\rm FP}(n,d;p).$$
\end{theorem}
 \begin{theorem}\label{main-theorem-exact-bound}If $n=2$, $d$ is even, $p\equiv-1({\rm mod}\ d)$, and $p>d^3$, then
$${\rm GNP}(n,d;p)={\rm FP}(n,d;p).$$
\end{theorem}
From the above theorem one can deduce the following corollary.
\begin{corollary}\label{wan-conjecture-minus-1}If $n=2$ and $d$ is even, then
$$\lim_{\stackrel{p\equiv-1({\rm mod}\ d)}{p \rightarrow\infty}}{\rm GNP}(n,d;p)={\rm HP}(n,d).$$
\end{corollary}
\subsection{Conjectures}
\paragraph{}
Wan [2004] put forward the following conjecture.
\begin{conjecture}[Wan's asymptotic $p$-adic Riemann Hypothesis]\label{wan-conjecture} The following identity holds:
$$\lim_{p\rightarrow\infty}{\rm GNP}(n,d;p)={\rm HP}(n,d).$$
\end{conjecture}
The $T$-adic version of the above conjecture was put forward in [Liu and Wan 2009].
We put forward the following conjecture.
\begin{conjecture}[$p$-adic Riemann Hypothesis]\label{conjecture}If $p$ is sufficiently large, then
$${\rm GNP}(n,d;p)={\rm FP}(n,d;p).$$
\end{conjecture}
\subsection{Organization of Present Paper}
\paragraph{}
In \S2, we prove the nonnegativity and symmetry of Frobenius numbers, as well as their triviality in trivial case.

\paragraph{}
In \S3, we prove the privilege of Frobenius polygon over Hodge polygon, i.e.
$${\rm FP}(n,d;p)\geq{\rm HP}(n,d).$$

\paragraph{}
In \S4, we define the premium polygon ${\rm PP}(n,d;p)$ of ${\rm Simp}(n,d)$,
and prove that, if $p>2d$, then
$${\rm PP}(n,d;p)={\rm FP}(n,d;p).$$
\paragraph{}
In \S5, we define, for each Frobenius vertex $k$, a system of twisted Hasse polynomials
$${\rm TH}^{(a)}(k)\in{\mathbb F}_p[a_w\mid w\in{\rm Simp}(n,d)\cap{\mathbb Z}^n],\ a=1,2,\cdots.$$
We prove that, when viewed as polynomials over
$${\mathbb F}_p[a_w\mid \deg w=1],$$
the twisted Hasse polynomials have degrees less than $p^a$ if $p>d^{n+1}$.

\paragraph{}
In \S6,
we prove that, if $p>2d$, then
$${\rm GNP}(n,d;p)\geq{\rm PP}(n,d;p).$$
We also prove that, if $p>2d$, and if for each Frobenius vertex $k\leq{d+n-2\choose n}$, there exist a positive integer $a$ and a polynomial
$$f(x)\in{\mathbb F}_{p^a}[x_1,\cdots,x_n]_{(d)}$$
whose leading form is smooth such that
$${\rm TH}^{(a)}(k)(f)\neq0,$$
 then
$${\rm GNP}(n,d;p)={\rm PP}(n,d;p)\text{ on }[0,{d+n-2\choose n}].$$

\paragraph{}
In \S7, we assume that $n=2$, and $d$ is even, and $p\equiv-1 ({\rm mod}\ d)$.
We prove that, if $p>d^3$, then for each Frobenius vertex $k\leq{d\choose 2}$, and for each $a>1$, there exists a polynomial
$$f(x)\in{\mathbb F}_{p^a}[x_1,x_2]_{(d)}$$
whose leading form is smooth such that
$${\rm TH}^{(a)}(k)(f)\neq0.$$
\paragraph{}
{\bf Acknowledgements.} The first author thanks Daqing Wan for his helpful suggestions and careful reading of the manuscript.
\section{Frobenius Numbers}
In this section we define Frobenius numbers and prove their nonnegativity.
To this end, we review Hodge numbers briefly and introduce arithmetic Hodge sums.
\subsection{Hodge numbers}
Hodge numbers are symmetric in the following sense.
\begin{lemma}[symmetry]$$h_{nd-j}=h_j,\ j=0,1,\cdots,nd.$$\end{lemma}
\begin{proof}Obvious.\end{proof}
The first half Hodge numbers are increasing.
\begin{lemma}[monotonicity]\label{hodge-numbers-monotonicity}
$$h_1\leq h_2\leq\cdots\leq h_{[\frac{nd}{2}]}.$$\end{lemma}\begin{proof}To indicate the dependence of $h_j$ on $n$, we write
$h_j$ as $h_j^{(n)}$. Then
$$\sum_{j=0}^{nd}(h_j^{(n)}-h_{j-1}^{(n)})t^{\frac{j}{d}}
=\frac{(t^{\frac{1}{d}}-t)^n}{(1-t^{\frac{1}{d}})^{n-1}}
=\sum_{j=0}^{nd}(h_{j-1}^{(n-1)}-h_{j-d}^{(n-1)})t^{\frac{j}{d}}.$$
It follows that$$h_j^{(n)}-h_{j-1}^{(n)}=h_{j-1}^{(n-1)}-h_{j-d}^{(n-1)}.$$
By induction and by symmetry of Hodge numbers, $$h_j^{(n)}-h_{j-1}^{(n)}=h_{j-1}^{(n-1)}-h_{j-d}^{(n-1)}\geq0\text{ if }j\leq\frac{(n-1)d}{2}+1,$$
and
$$h_j^{(n)}-h_{j-1}^{(n)}
=h_{(n-1)d-j+1}^{(n-1)}-h_{j-d}^{(n-1)}\geq0\text{ if }\frac{(n-1)d}{2}+1<j\leq\frac{nd}{2}.$$
\end{proof}
\subsection{Arithmetic  Hodge sums}
 \begin{definition}[arithmetic  Hodge sums]The arithmetic  Hodge sums of ${\rm Simp}(n,d)$
are the numbers
$$H_i:=\sum_{j=0}^{[\frac{i}{d}]}h_{i-jd}, \ i=0,1,\cdots,nd.$$\end{definition}
By the above definition, Hodge numbers are arithmetic differences of arithmetic Hodge sums in the following sense:
$$H_i-H_{i-d}=h_i,\  i=0,1,\cdots,nd.$$
The numbers $H_i$'s with $i\equiv0 ({\rm mod}\ d)$ will play trivial roles in defining Frobenius numbers. So we concentrate on the numbers $H_i$'s with $i\not\equiv0 ({\rm mod}\ d)$, which we call basic arithmetic Hodge sums.
The basic arithmetic Hodge sums are symmetric in the following sense.
\begin{lemma}[symmetry] If $i\not\equiv0 ({\rm mod}\ d)$, then
$$H_i+H_{nd-i-d}=\frac{(d-1)^n-(-1)^n}{d}.$$\end{lemma}
\begin{proof}
By symmetry of Hodge numbers,
$$H_{nd-i-d}=\sum_{j=0}^{n-2-[\frac{i}{d}]}h_{nd-i-d-jd}=\sum_{j=1}^{n-1-[\frac{i}{d}]}h_{i+jd}.$$
So
$$H_i+H_{nd-i-d}=\sum_{j=-[\frac{i}{d}]}^{n-1-[\frac{i}{d}]}h_{i+jd}=\sum_{\stackrel{u_1,\cdots,u_n=1}{u_1+\cdots+u_n\equiv i({\rm mod } \ d)}}^{d-1}1 .$$
We can prove that
$$\sum_{\stackrel{u_1,\cdots,u_n=1}{u_1+\cdots+u_n\equiv i({\rm mod } \ d)}}^{d-1}1 =\frac{(d-1)^n-(-1)^n}{d}.
$$
The lemma now follows.
\end{proof}
Besides symmetry,  the basic arithmetic Hodge sums also have monotonicity.
\begin{lemma}[monotonicity]\label{monotonicity}If $0\leq i\leq k\leq nd$ with $i,k\not\equiv0 ({\rm mod}\ d)$, then
$H_i\leq H_k$.
\end{lemma}
\begin{proof}We may assume that $n\geq2$. If
$k\leq[\frac{nd}{2}]$, then $H_i\leq H_k$ by the monotonicity of the first half Hodge numbers.
We now assume that $k\geq[\frac{nd}{2}]+1$. By symmetry of arithmetic Hodge sums,
it suffices to show that
$$H_{nd-k-d}\leq H_{nd-i-d}.$$
Let
$i_0\equiv i({\rm mod}\ d)$ be the integer such that
$$[\frac{nd}{2}]-d< nd-i_0-d\leq[\frac{nd}{2}].$$
It suffices to show that
$$H_{nd-k-d}\leq H_{nd-i_0-d},$$
which follows from the monotonicity of the first half Hodge numbers, as
$$nd-k-d\leq nd-[\frac{nd}{2}]-1-d\leq[\frac{nd}{2}]-d< nd-i_0-d\leq[\frac{nd}{2}].$$
The lemma is proved.\end{proof}The basic arithmetic Hodge sums are symmetric in the sense that
$$H_i+H_{nd-i-d}
=\frac{(d-1)^n-(-1)^n}{d}\text{ if }i\not\equiv0 ({\rm mod}\ d).$$
They are increasing in the sense that
$$H_i\leq H_k,
\text{ if }0\leq i\leq k\leq nd\text{ and }i,k\not\equiv0 ({\rm mod}\ d).$$
\subsection{Frobenius Numbers}
\begin{definition}[conjugates of Frobenius]For $i=0,1,\cdots,nd$, the $i$-th conjugate $\sigma_i$ of Frobenius
 is the permutation
 on ${\mathbb Z}/(d)$
such that $$i-\sigma_i(l)\equiv p^{-1}(i -l)({\rm mod}\ d),$$
where ${\mathbb Z}/(d)$ is identified with $\{0,1,2,\cdots,d-1\}$ if $(p-1)i\not\equiv0 ({\rm mod}\ d)$, and is identified with
$\{1,2,\cdots,d\}$ if $(p-1)i\equiv0 ({\rm mod}\ d)$.
\end{definition}
By the above definition, the following diagram commutes:\\
\begin{center}
\begin{tabular}{ccc}
                                      ${\mathbb Z}/(d)$ & $\stackrel{l\mapsto i-l}{\rightarrow}$ & ${\mathbb Z}/(d)$ \\
                                      $\sigma_i\downarrow$  &  &$\downarrow p^{-1}$  \\
                                     ${\mathbb Z}/(d)$ & $\stackrel{l\mapsto i-l}{\rightarrow}$ & ${\mathbb Z}/(d)$ \\
                                    \end{tabular}\end{center}
\begin{definition}[Frobenius numbers]Frobenius numbers of ${\rm Simp}(n,d)$ are the numbers:
$$h_{j,0}:=H_j-H_{j-\sigma_j(0)},$$
$$h_{j,1}:=H_{j-\sigma_j(0)}-H_{j-d},
$$
where $j=0,1,\cdots,nd$.\end{definition}
Frobenius numbers form decompositions of Hodge numbers in the following sense.
\begin{lemma}If $j=0,1,\cdots,nd$, then
$h_j=h_{j,0}+h_{j,1}$.
\end{lemma}
\begin{proof}Obvious.\end{proof}
Frobenius numbers are trivial in trivial case in the following sense.
\begin{lemma}If $(p-1)j\equiv0({\rm mod }\ d)$, then
$h_{j,1}=0$.
\end{lemma}
\begin{proof}We can show that
 $$\sigma_j(0)=d\text{ if }(p-1)j\equiv0({\rm mod }\ d).$$
The lemma now follows.\end{proof}
Frobenius numbers are nonnegative.
\begin{lemma}If $j=0,1,\cdots,nd$ and $\epsilon=0,1$, then
$h_{j,\epsilon}\geq0$.
\end{lemma}
\begin{proof}This follows from the monotonicity of basic arithmetic Hodge sums.\end{proof}
Besides nonnegativity, Frobenius numbers also have symmetry in the following sense.
\begin{lemma}If $(p-1)j\not\equiv0({\rm mod }\ d)$, then
$$h_{nd-j,1-\epsilon}=h_{j,\epsilon}.$$
\end{lemma}
\begin{proof}It suffices to show that
$$h_{nd-j,1}=h_{j,0}.$$
We can show that
$$\sigma_{nd-j}(0)=d-\sigma_j(0).
$$
So by symmetry of arithmetic Hodge numbers,
$$h_{nd-j,1}=H_{nd-j-\sigma_{nd-j}(0)}-H_{nd-j-d}
=H_j-H_{j-\sigma_j(0)}=h_{j,0}.
$$
The lemma is proved. \end{proof}
\section{Frobenius Polygon}
In this section we prove that Frobenius polygon lies above Hodge polygon.
To this end, we introduce fitted Frobenius polygon.
We prove that at any vertex of Frobenius polygon there is a fitted Frobenius polygon that fits Frobenius polygon at
that vertex. We then prove the special vertex of fitted Frobenius polygon lies  above Hodge polygon. The privilege of Frobenius polygon over Hodge polygon follows immediately.
\subsection{Fitted  Frobenius polygon}
\begin{definition}[fitted Frobenius slopes]
The $i$-fitted Frobenius slopes are the numbers:
$$\varpi_j^{(i)}:=\frac{1}{p-1}(\lceil\frac{pj-i}{d}\rceil-\lceil\frac{j-i}{d}\rceil),\ j=0,1,\cdots,i.
$$\end{definition}
\begin{definition}The $i$-fitted Frobenius polygon ${\rm FP}^{(i)}(n,d;p)$ is a concave polygon on the right half plane with initial point $(0,0)$ whose slope distribution is listed in the following table:

 \begin{tabular}{|c|c|c|}
                                                      \hline
                                                      slope & multiplicity &variation\\
                                                      \hline
                                                      $\varpi_j^{(i)}$ & $h_j$ &$j=0,1,\cdots,i$\\
                                                      \hline
                                                    \end{tabular}

\end{definition}
\begin{lemma}If $k=\sum_{2j-\epsilon\leq2i-\iota}h_{j,\epsilon}$ is a Frobenius vertex, then
$${\rm FP}(n,d;p)(k)
={\rm FP}^{(i)}(n,d;p)(k).
$$
                               \end{lemma}
\begin{proof}It is easy to see that
\begin{eqnarray*}
 & & \sum_{j=0}^{i}(\lceil\frac{(p-1)j}{d}\rceil+\lceil\frac{j-i}{d}\rceil-\lceil\frac{pj-i}{d}\rceil)h_j
 \\
 & =& \sum_{j=0}^{i-1}(\lceil\frac{(p-1)j}{d}\rceil+\lceil\frac{j-i}{d}\rceil-\lceil\frac{pj-i}{d}\rceil)h_j
 \\
   &=& \sum_{j=0}^{i-1}(\lceil\frac{(p-1)j}{d}\rceil+\lceil\frac{j-(i-1)}{d}\rceil-\lceil\frac{pj-(i-1)}{d}\rceil)h_j+h_{i,1}.
\end{eqnarray*}
It follows that
$$\sum_{j=0}^ih_{j,1}=\sum_{j=0}^{i}(\lceil\frac{(p-1)j}{d}\rceil+\lceil\frac{j-i}{d}\rceil-\lceil\frac{pj-i}{d}\rceil)h_j.$$
Hence\begin{eqnarray*}
{\rm FP}(n,d;p)(k)&=& \frac{1}{p-1}(\sum_{2j-\epsilon\leq2i-\iota}h_{j,\epsilon}\lceil \frac{(p-1)j}{d}\rceil
-\sum_{j=0}^ih_{j,1}) \\
       &=& \sum_{2j-\epsilon\leq2i-\iota}\frac{h_{j,\epsilon}}{p-1}(\lceil\frac{pj-i}{d}\rceil-\lceil\frac{j-i}{d}\rceil)\\
&=&{\rm FP}^{(i)}(n,d;p)(k).
     \end{eqnarray*}
The lemma is proved.\end{proof}
\begin{lemma}If $k=\sum_{j=0}^ih_j$ is a Hodge vertex, then
$${\rm FP}^{(i)}(n,d;p)(k)-{\rm HP}(n,d)(k)=\sum_{\stackrel{l=0}{l\not\equiv i({\rm mod}\ d)}}^{d-1}\frac{(d-1-l)}{(p-1)d}(H_{i-l}-H_{i-\sigma_i(l)}).
$$\end{lemma}
\begin{proof}This follows from direct computations.\end{proof}
\begin{lemma}
If $k=\sum_{j=0}^{i-1}h_j+h_{i,1}$ with $h_{i,1}\neq0$, then
$${\rm FP}^{(i)}(n,d;p)(k)-{\rm HP}(n,d)(k)=\sum_{\stackrel{l=1}{\stackrel{l\not\equiv i({\rm mod}\ d)}{l\not\equiv (1-p)i({\rm mod}\ d)}}}^{d-1}\frac{(d-1-l)}{(p-1)d}(H_{i-l}-H_{i-\sigma_i(l)}).
$$\end{lemma}
\begin{proof}This follows from direct computations too.\end{proof}
\subsection{Privilege of Frobenius Polygon}
\begin{theorem}$${\rm FP}(n,d;p)\geq{\rm HP}(n,d).$$
                       \end{theorem}
\begin{proof}This follows from the lemmas in the previous subsection and the monotonicity of basic arithmetic Hodge sums.\end{proof}
\section{Premium Polygon}
In this section we define the premium polygon ${\rm PP}(n,d;p)$ of ${\rm Simp}(n,d)$, which serves as a bridge between Frobenius polygon and the generic Newton polygon. We prove that, if $p>2d$, then the premium polygon coincides with
Frobenius polygon, i.e.
$${\rm PP}(n,d;p)={\rm FP}(n,d;p).$$
\subsection{Twisted Permutations and their Premium}
\begin{definition}For an $a$-tuple $A=(A_l)_{l\in{\mathbb Z}/(a)}$ of sets, the twist $A(1)$ of $A$
 is define by the formula
$$A(1)_l:=A_{l-1},\ \forall l\in{\mathbb Z}/(a).$$\end{definition}
\begin{definition}For  $a$-tuples $A$ and $B$ of sets, we define the set of isomorphisms from $A$ to $B$ to be the set
$${\rm Iso}(A,B):=\prod_{l=0}^{a-1}{\rm Iso}(A_l,B_l).$$\end{definition}
\begin{definition}For  $a$-tuples $A$, we define the set of twisted permutations of $A$ to be the set
$${\rm Sym}^{(1)}(A):={\rm Iso}(A,A(1)).$$\end{definition}
It is easy to see that
$${\rm Sym}^{(1)}(A^a)=\left({\rm Sym}^{(1)}(A)\right)^a.$$
We call points in ${\mathbb N}^n$ (positive) lattices points, and call points in $(0,d)^n\cap{\mathbb Z}^n$
fundamental lattice points.
\begin{definition}For a lattice point $u\in{\mathbb N}^n$, the degree of $u$ with respect to ${\rm Simp}(n,d)$ is
$$\deg(u):=\frac{u_1+u_2+\cdots+u_n}{d}.$$
\end{definition}
\begin{definition}Let $A$ be an $a$-tuple of subsets of fundamental lattice points. For $\tau\in{\rm Sym}^{(1)}(A)$, we define Frobenius
premium of $\tau$ to be the number
$${\rm Prem}(\tau):=\frac{1}{a(p-1)}\sum_{l=0}^{a-1}\sum_{u\in A_l}\lceil \deg (pu-\tau_l(u)) \rceil.$$
\end{definition}
\subsection{Premium Polygon}
\begin{definition}For $a$-tuple $A$ of subsets of fundamental lattice points, we define Frobenius
premium of $A$ to be the number
$${\rm Prem}(A):=\min_{\tau\in{\rm Sym}^{(1)}(A)}{\rm Prem}(\tau).$$
\end{definition}
It is easy to see that, if $A$ is a subset of fundamental lattice points, then
$${\rm Prem}(A^a)={\rm Prem}(A).$$
\begin{definition}For $a$-tuple $A$ of subsets of fundamental lattice points, we define ${\rm Sym}_0^{(1)}(A)$
the subset of ${\rm Sym}^{(1)}(A)$ consisting of elements of minimal premium, i.e.
$${\rm Sym}_0^{(1)}(A):=\{\tau\in{\rm Sym}^{(1)}(A)\mid {\rm Prem}(\tau)={\rm Prem}(A)\}.$$
\end{definition}
It is easy to see that, if $A$ is a subset of fundamental lattice points, then
$${\rm Sym}_0^{(1)}(A^a)=\left({\rm Sym}_0^{(1)}(A)\right)^a.$$
\begin{definition}For two subsets $A$ and $B$ of fundamental lattices points, we denote by ${\rm Iso}_{\deg}(A,B)$
the set of degree-preserving isomorphisms from $A$ to $B$, i.e.
$${\rm Iso}_{\deg}(A,B):=\{\alpha\in{\rm Iso}(A,B)\mid \deg(\alpha(u))=\deg(u),\ \forall u\in A\}.$$ \end{definition}
\begin{definition}For  $a$-tuples $A$ and $B$ of subsets of fundamental lattices points, we define the set of degree-preserving isomorphisms from $A$ to $B$ to be the set
$${\rm Iso}_{\deg}(A,B):=\prod_{l=0}^{a-1}{\rm Iso}_{\deg}(A_l,B_l).$$\end{definition}
\begin{lemma}Let $A$ and $B$ be $a$-tuples of subsets of fundamental lattices points.
If there is a degree-preserving isomorphism from $A$ to $B$, then
$${\rm Prem}(A)={\rm Prem}(B).$$\end{lemma}
\begin{proof}
We can prove that, if  $\alpha\in{\rm Iso}_{\deg}(A,B)$, then
$${\rm Prem}(\alpha(1)\circ\tau\circ\alpha^{-1})={\rm Prem}(\tau), \forall \tau\in{\rm Sym}^{(1)}(A),$$
where $\alpha(1)\in{\rm Iso}_{\deg}(A(1),B(1))$ is the isomorphism induced by $\alpha$.
The lemma now follows.\end{proof}
\begin{definition}For $k=0,1,\cdots,nd$, we define ${\rm SF}(k)$ to be the set of subsets of fundamental lattice points of cardinality $k$, i.e.
$${\rm SF}(k):=\{A\subseteq(0,d)^n\cap{\mathbb Z}^n:\  |A|=k\}.$$ \end{definition}
\begin{definition}For $k=0,1,\cdots,(d-1)^n$, we define the premium at $k$ to be the number
$${\rm Prem}(k):=\min_{A\in{\rm SF}(k)^a}{\rm Prem}(A).$$ \end{definition}
\begin{definition}We define the premium polygon ${\rm PP}(n,d;p)$ of ${\rm Simp}(n,d)$ to be the convex hull
in ${\mathbb R}^n$ of the points
$$(k,{\rm Prem}(k)),\ k=0,1,\cdots,(d-1)^n.$$\end{definition}
\subsection{Coincidence with Frobenius Polygon}
\begin{definition}For finite subset $A\subseteq{\mathbb N}^n$, we define the degree of $A$ to be
$$\deg(A):=\sum_{u\in A}\deg (u).$$
\end{definition}
It is easy to see that
 $$\min_{A\in{\rm SF}(k)}\deg (A)={\rm HP}(n,d)(k).$$
\begin{definition}We define ${\rm SF}_0(k)$ to be the subset of ${\rm SF}(k)$ consisting of elements of minimal degree, i.e.
$${\rm SF}_0(k)=\{A\in{\rm SF}(k)\mid \deg A={\rm HP}(n,d)(k)\}.$$ \end{definition}
\begin{definition}For $i=0,1,\cdots,nd$, we write
$$W_i:=\{u\in(0,d)^n\cap{\mathbb Z}^n\mid \deg(u)=\frac{i}{d}\}.$$
\end{definition}
It is easy to see that, if $k=\sum_{j=0}^ih_j$, then
$$A\in{\rm SF}_0(k)\Leftrightarrow A=\cup_{j=0}^iW_j.$$
It is also easy to see that, if $\sum_{j=0}^{i-1}h_j\leq k\leq\sum_{j=0}^ih_j$, then
$$A\in{\rm SF}_0(k)\Leftrightarrow\cup_{j=0}^{i-1}W_j\subseteq A\subseteq\cup_{j=0}^iW_j.$$
In particular, there are degree-preserving isomorphisms between elements of ${\rm SF}_0(k)$.
\begin{lemma}\label{premium-lower-bound}If $A\in{\rm SF}_0(k)^a$,
then
$${\rm Prem}(A)\geq {\rm FP}(n,d;p)(k).$$
\end{lemma}
\begin{proof}Let $\tau\in {\rm Sym}^{(1)}(A)$.
Firstly we assume that
$$\sum_{j=0}^{i-1}h_j+h_{i,1}<k\leq \sum_{j=0}^ih_j.$$
Then
\begin{eqnarray*}
{\rm Prem}(\tau)  &\geq&\frac{1}{a(p-1)}\sum_{l=0}^{a-1}\sum_{u\in A_l}
\left(\lceil\deg(pu)-\frac{i}{d}\rceil-\lceil\deg u-\frac{i}{d}\rceil\right)\\
&\geq& \sum_{2j-\epsilon\leq 2i-1}h_{j,\epsilon}\varpi_j^{(i)}+(k-\sum_{2j-\epsilon\leq 2i-1}h_{j,\epsilon})\varpi_{i,0}\\
&\geq&{\rm FP}(n,d;p)(k)
\end{eqnarray*}
Secondly we assume that
$$\sum_{j=0}^{i-1}h_j<k\leq \sum_{j=0}^{i-1}h_j+h_{i,1}.$$
Then
\begin{eqnarray*}
&&\#\{u\in A_l\mid \lceil\deg(pu-\tau_l(u))\rceil
>\lceil\deg(pu)-\frac{i}{d}\rceil-\lceil\deg\tau_l(u)-\frac{i}{d}\rceil\}\\
    &\geq& \#\{u\in A_l\mid\{\deg(pu)+\frac{d-1-i}{d}\}>\{\deg\tau_l(u)+\frac{d-1-i}{d}\}\} \\
    &\geq&\#\{u\in A_l\mid\deg u<\frac{i}{d},\ \{\deg(pu)\}=\frac{i}{d}\}\\
&&-\#\{u\in A_{l-1}\mid\deg u<\frac{i}{d},\ \{\deg u\}=\frac{i}{d}\}-(k-\sum_{j=0}^{i-1}h_j)\\
 &\geq&\sum_{j=0}^{i-1}h_j+h_{i,1}-k.
\end{eqnarray*}
It follows that
\begin{eqnarray*}
    {\rm Prem}(\tau)&\geq&\frac{1}{a(p-1)}\sum_{l=0}^{a-1}\sum_{u\in A_l}\left(\lceil\deg(pu)-\frac{i}{d}\rceil-\lceil\deg\tau_l(u)-\frac{i}{d}\rceil\right)\\
&&+\frac{1}{p-1}(\sum_{j=0}^{i-1}h_j+h_{i,1}-k)\\
  & \geq&\sum_{2j-\epsilon\leq 2i-1}h_{j,\epsilon}\varpi_j^{(i)} -(\sum_{j=0}^{i-1}h_j+h_{i,1}-k)\varpi_{i,1}\\
  & \geq&{\rm FP}(n,d;p)(k).
\end{eqnarray*}
The lemma now follows.
\end{proof}
Similarly, we can prove two following lemmas.
\begin{lemma}\label{premium-strict-lower-bound}If $A\in{\rm SF}(k)^a\setminus{\rm SF}_0(k)^a$,
then
$${\rm Prem}(A)\geq {\rm FP}(n,d;p)(k)+\frac{1}{a(p-1)}\left([\frac{p}{d}]-1\right).$$
\end{lemma}
\begin{lemma}Let $k=\sum_{2j-\epsilon\leq 2i-\iota}h_{j,\epsilon}$, $A\in{\rm SF}_0(k)^a$, and $\tau\in{\rm Sym}^{(1)}(A)$. Then
$${\rm Prem}(\tau)={\rm FP}(n,d;p)(k)$$ if and only if
$$\{\deg\tau_l(u)+\frac{d-1-i}{d}\}\geq\{p\deg u+\frac{d-1-i}{d}\}, \forall u\in A_l,\ \forall l\in{\mathbb Z}/(a).$$
\end{lemma}
We can show that, if $k=\sum_{2j-\epsilon\leq 2i-\iota}h_{j,\epsilon}$, and $A\in{\rm SF}_0(k)^a$,
 then there exists $\tau\in{\rm Sym}^{(1)}(A)$ such that
$$\{\deg\tau_l(u)+\frac{d-1-i}{d}\}\geq\{p\deg u+\frac{d-1-i}{d}\}, \forall u\in A_l,\ \forall l\in{\mathbb Z}/(a).$$
Therefore we have the following corollaries.
\begin{corollary}If $k=\sum_{2j-\epsilon\leq 2i-\iota}h_{j,\epsilon}$, and $A\in{\rm SF}_0(k)^a$,
then
$${\rm Prem}(A)={\rm FP}(n,d;p)(k),$$
and ${\rm Sym}_0^{(1)}(A)$ consists of twisted permutations $\tau$ such that
$$\{\deg\tau_l(u)+\frac{d-1-i}{d}\}\geq\{p\deg u+\frac{d-1-i}{d}\}, \forall u\in A_l,\ \forall l\in{\mathbb Z}/(a).$$
\end{corollary}
\begin{corollary}If $p>2d$, then the premium polygon coincides with Frobenius polygon, i.e.
$${\rm PP}(n,d;p)={\rm FP}(n,d;p).$$
\end{corollary}

\section{Twisted Hasse Polynomials}
In this section, we define, for each Frobenius vertex $k$, a system of twisted Hasse polynomials
$${\rm TH}^{(a)}(k)\in{\mathbb F}_p[a_w\mid w\in{\rm Simp}(n,d)\cap{\mathbb Z}^n],\ a=1,2,\cdots.$$
We prove that,
when viewed as polynomials over
$${\mathbb F}_p[a_w\mid \deg w=1],$$
the twisted Hasse polynomials have degree less than $p^a$ if $p>d^{n+1}$.

\subsection{Polynomials of Lattice Points}
\begin{definition}An integral section on ${\rm Simp}(n,d)$ is a map
$$s:{\rm Simp}(n,d)\cap{\mathbb Z}^n\rightarrow{\mathbb N}.$$
\end{definition}
\begin{definition}The integral bundle ${\rm Bund}(n,d)$ on ${\rm Simp}(n,d)$ is the set of integral sections on ${\rm Simp}(n,d)$.
\end{definition}
\begin{definition}For $s\in{\rm Bund}(n,d)$, we define the degree of $s$ to be the number
$$\deg (s):=\sum_{w\in{{\rm Simp}(n,d)\cap{\mathbb Z}^n}}s(w).$$\end{definition}
Let ${\mathbb Z}_p$ be the ring of $p$-adic integers, and let
$$E(x)=\exp(\sum_{i=0}^{\infty}\frac{x^{p^i}}{p^i})\in{\mathbb Z}_p[[x]]$$
be the Artin-Hasse exponential.
Write
$$E(x)=\sum_{j=0}^{+\infty}\lambda_jx^j,$$
and call
$$\overline{\lambda}_j=\lambda_j+{\mathbb Z}_p/(p)\in{\mathbb F}_p,\ \forall j\in{\mathbb N}$$
Artin-Hasse coefficients in characteristic $p$.

\begin{definition}For $s\in{\rm Bund}(n,d)$, we define the monomial representation
 of $s$ to be the monomial
$${\rm Mono}(s):=\prod_{w\in{{\rm Simp}(n,d)\cap{\mathbb Z}^n}}\overline{\lambda}_{s(w)}a_w^{s(w)}\in{\mathbb F}_p[a_w\mid w\in{\rm Simp}(n,d)\cap{\mathbb Z}^n].$$
\end{definition}
Notice that the first $p$ Artin-Hasse coefficients $\overline{\lambda}_0,\overline{\lambda}_1,\cdots,\overline{\lambda}_{p-1}$  are all nonzero. So
$$\deg{\rm Mono}(s)=\deg(s)\text{ if }\deg(s)<p.$$
\begin{definition}For $s\in{\rm Bund}(n,d)$, we define the lattice point represented by $s$ to be the vector
$${\rm vec}(s):=\sum_{w\in{{\rm Simp}(n,d)\cap{\mathbb Z}^n}}s(w)w.$$
\end{definition}
It is easy to see that
$$\min_{{\rm vec}(s)=u}\deg(s)=\lceil\deg(u)\rceil.$$
\begin{definition}For $u\in{\mathbb N}^n$, we define ${\rm Bund}(n,d)(u)$ to be the set of sections representing $u$ with minimal degree, i.e.
$${\rm Bund}(n,d)(u):=\{s\in{\rm Bund}(n,d)\mid {\rm vec}(s)=u,\ \deg(s)=\lceil\deg(u)\rceil\}.$$\end{definition}
\begin{definition}The polynomial of lattice point $u\in{\mathbb N}^n$ is defined to be the polynomial
$${\rm Poly}(u):=\sum_{s\in{\rm Bund}(n,d)(u)}{\rm Mono}(s).$$
\end{definition}
\subsection{Twisted Hasse Polynomials}
\begin{definition}For an $a$-tuple $A$ of subsets of fundamental lattices points, we define
$${\rm Poly}(A):=\sum_{\tau\in{\rm Sym}_0^{(1)}(A)}(-1)^{{\rm sgn} \tau}\prod_{l=1}^a\prod_{u\in A_l}{\rm Poly}(pu-\tau_l(u))^{p^{l-1}}.$$
\end{definition}
It is easy to see that, if $A$ is a subset of fundamental lattices points, then
$${\rm Poly}(A^a)=\prod_{l=1}^a{\rm Poly}(A)^{p^{l-1}}.$$
\begin{definition}For a Frobenius vertex $k=\sum_{2j-\epsilon\leq 2i-\iota}h_{j,\epsilon}$, the twisted Hasse polynomials at $k$ are the polynomials
$${\rm TH}^{(a)}(k):=
\sum_{A\in{\rm SF}_0(k)^a}{\rm Poly}(A),\ a=1,2,\cdots.$$
\end{definition}
It is easy to see that, if $k=\sum_{j=0}^ih_j$ is a Hodge vertex, then
$${\rm TH}^{(a)}(k)=\prod_{l=1}^a{\rm Poly}(\cup_{j=0}^iW_j)^{p^{l-1}}.$$
\begin{lemma}If  $p>d^{n+1}$, and $k=\sum_{2j-\epsilon\leq2i-\iota}h_{j,\epsilon}$ is a Frobenius vertex, then,
 when viewed as a polynomial over
$${\mathbb F}_p[a_w\mid w\in{\rm Simp}(n,d)\cap{\mathbb Z}^n],$$
the twisted polynomial ${\rm TH}^{(a)}(k)$
has degree less than $p^a$.
\end{lemma}
\begin{proof}
It suffices to show that,
if $u\in{\mathbb N}^n$ and $s\in {\rm Bund}(n,d)(u)$, then
$$\sum_{\deg w<1}s(w)<d.$$
Indeed, as
$$u=\sum_{w\in{\rm Simp}(n,d)\cap{\mathbb Z}^n}s(w)w,$$
and
$$\deg(s)=\lceil\deg(u)\rceil,$$
we have\begin{eqnarray*}
                 \deg(u) &=& \sum_{\stackrel{w\in{\rm Simp}(n,d)\cap{\mathbb Z}^n}{\deg(w)=1}}s(w)
+\sum_{\stackrel{w\in{\rm Simp}(n,d)\cap{\mathbb Z}^n}{\deg(w)<1}}s(w)\deg(w) \\
                  &=& \sum_{w\in{\rm Simp}(n,d)\cap{\mathbb Z}^n}s(w)
+\sum_{\stackrel{w\in{\rm Simp}(n,d)\cap{\mathbb Z}^n}{\deg(w)<1}}s(w)(\deg(w)-1) \\
                  &\leq& \lceil\deg(u)\rceil
-\frac{1}{d}\sum_{\deg w<1}s(w).
               \end{eqnarray*}
The lemma now follows.
\end{proof}
\section{Generic Newton Polygon}
In this section we prove that, if $p>2d$, then
$${\rm GNP}(n,d;p)\geq{\rm PP}(n,d;p).$$
We also prove that, if $p>2d$, and if for each Frobenius vertex $k\leq{d+n-2\choose n}$, there exist a positive integer $a$ and a polynomial
$$f(x)\in{\mathbb F}_{p^a}[x_1,\cdots,x_n]_{(d)}$$
whose leading form is smooth such that
$${\rm TH}^{(a)}(k)(f)\neq0,$$
 then
$${\rm GNP}(n,d;p)={\rm PP}(n,d;p)\text{ on }[0,{d+n-2\choose n}].$$
\subsection{Trace Formula of $a$-th power of Dwork Operator}
\paragraph{}
In this subsection, we review the $p$-adic theory developed by Dwork and Sperber, especially a trace formula relating
the characteristic polynomial of the $a$-power Dwork operator to $L$-function of exponential sums.

Let $q$ be a power of $p$, $\mu_{q-1}$  the group of $(q-1)$-th roots of unity, and ${\mathbb Z}_q={\mathbb Z}_p[\mu_{q-1}]$.
Let $\pi$ be the Artin-Hasse uniformizer of $\zeta_p-1$ in the sense that
$$E(\pi)=\zeta_p.$$
Let
$$f(x)=\sum_{u\in{\rm Simp}(n,d)\cap{\mathbb Z}^n}a_ux^u\in{\mathbb F}_q[x_1,\cdots,x_n]_{(d)},$$
where
$x^u=x_1^{u_1}\cdots x_n^{u_n}$.
We often identify $f$ with the vector
$$(a_u)_{u\in{\rm Simp}(n,d)\cap{\mathbb Z}^n}.$$
And, for brevity, we often identify elements of finite fields with their Teichm\"{u}ller lifts.
Write
$$E_f(x)=\prod_{\deg u\leq1}E(\pi a_ux^u).$$
The function $E_f$ lies in the ${\mathbb Z}_q[\pi^{\frac{1}{d}}]$-module
$$L=\left\{
\sum_{u\in{\mathbb N}^n}
b_u\pi^{\deg u}x^u\mid b_u\in{\mathbb Z}_q[\pi^{\frac{1}{d}}]\right\},$$
and acts on that module by multiplication. Let
$\phi$ be the operator on $L$ defined by the formula
$$\phi(\sum_{u\in{\mathbb N}^n}
b_u\pi^{\deg u}x^u)=\sum_{u\in{\mathbb N}^n}
b_{pu}\pi^{\deg u}x^u.$$
Let
$$B=\left\{
\sum_{u\in{\mathbb N}^n}
b_u\pi^{\deg u}x^u\in L\mid
\lim_{\deg u\rightarrow+\infty}{\rm ord}_p(b_u)=+\infty\right\}.$$
It is easy to see that
$B$ is closed under the composition $\phi\circ E_f$.
 So we regard $\phi\circ E_f$ as an operator on $B$.
Frobenius element of ${\rm Gal}({\mathbb F}_q/{\mathbb F}_p)$, which is denoted as $\sigma$, also acts on $B$
naturally. Let $\partial$ be the map
$$\partial:\oplus_{k=1}^n\pi^{\frac{1}{d}}x_kB\rightarrow B $$
defined by the formula
$$\partial(g_1,g_2,\cdots,g_n)=\sum_{k=1}^n(x_k\frac{\partial}{\partial x_k}+f_k)g_k,$$
where $f_k$'s are defined by the formula
$$d\log \widehat{E}_f(x)=\sum_{k=1}^nf_k\frac{dx_k}{x_k},\ \widehat{E}_f(x)=\prod_{j=0}^{+\infty}E_f^{\sigma^j}(x^{p^j}).$$
It is easy to see that the operators  $\phi\circ E_f$ and $\sigma$ live on the quotient
$$H_0(f)=B/{\rm Im}\partial.$$
So we now regard $\phi\circ E_f$ and $\sigma$ as an operators on $H_0(f)$.

 Sperber [1986] proved that, if   the leading form of $f$ is smooth,  then $H_0(f)$ is a free $Z_q[\pi^{\frac{1}{d}}]$-module of finite rank, and
$$\pi^{\deg u}x^u, \ u\in\{1,\cdots,d-1\}^n$$
represents a basis of $H_0(f)$.

The composition  $\phi_f=\sigma^{-1}\circ\phi\circ E_f$ is called Dwork operator. It acts on $H_0(f)$. Though Dwork operator is only linear over ${\mathbb Z}_p[\pi^{\frac{1}{d}}]$, its $a$-th power $\phi_f^a$ is linear over ${\mathbb Z}_q[\pi^{\frac{1}{d}}]$.

Sperber [1986] proved the following.
\begin{theorem}[trace formula of $a$-th power of Dwork operator]\label{trace-formula}If $q=p^a$, and
  $f\in{\mathbb F}_q[x_1,\cdots,x_n]$ has  smooth leading form, then
 $$L_{f,q}(t)^{(-1)^{n-1}}={\rm det}_{{\mathbb Z}_q[\pi^{\frac{1}{d}}]}(1-\phi_f^at\mid H_0(f)).$$
\end{theorem}

\subsection{Trace Formula of Dwork Operator}
\paragraph{}
In this subsection we prove a trace formula relating the characteristic polynomial
of Dwork operator to $L$-function of exponential sums.

We extend $\phi_f$ to $H_0(f)\otimes_{{\mathbb Z}_p[\pi^{\frac{1}{d}}]}{\mathbb Z}_q[\pi^{\frac{1}{d}}]$
via the formula
 $$\phi_f(v\otimes c):=\phi(v)\otimes c.$$
Then we identify
$H_0(f)\otimes_{{\mathbb Z}_p[\pi^{\frac{1}{d}}]}{\mathbb Z}_q[\pi^{\frac{1}{d}}]$
with $H_0(f)^{{\mathbb Z}/(a)}$
via the ${\mathbb Z}_q$-linear isomorphism
$$v\otimes c\mapsto c\cdot(\sigma^{i}(v))_{i\in{\mathbb Z}/(a)}.$$
Thus $\phi\circ E_f$ acts on  $H_0(f)^{{\mathbb Z}/(a)}$
via the formula
 $$\phi\circ E_f((v_i)_{i\in{\mathbb Z}/(a)}):=(\phi\circ E_f^{\sigma^i}(v_i))_{i\in{\mathbb Z}/(a)},$$
while
$\sigma^{-1}$ acts on  $H_0(f)^{{\mathbb Z}/(a)}$
via the formula
 $$\sigma^{-1}((v_i)_{i\in{\mathbb Z}/(a)}):=(v_{i-1})_{i\in{\mathbb Z}/(a)}.$$
Write
$$\varepsilon_{w,j}:=(\delta_{ij}\pi^{\deg w}x^w)_{i\in{\mathbb Z}/(a)}\in H_0^{{\mathbb Z}/(a)},$$
where $\delta_{ij}$ is Kronecker's $\delta$-function.
Then the matrix of
$\sigma^{-1}$ on $H_0(f)^{{\mathbb Z}/(a)}$ to the basis $((\varepsilon_{w,j})_{w\in\{1,2,\cdots,d-1\}^n})_{j\in{\mathbb Z}/(a)}$ is $$\left(
    \begin{array}{cccc}
      0 & \cdots &0 & I_{(d-1)^n} \\
     I_{(d-1)^n} & \cdots &0 & 0 \\
      \vdots & \ddots &\vdots & \vdots \\
      0 & \cdots & I_{(d-1)^n}&0 \\
    \end{array}
  \right)
$$
while the matrix of
$\phi\circ E_f$ on $H_0(f)^{{\mathbb Z}/(a)}$ to that basis is $$\left(
    \begin{array}{cccc}
      M &0 &\cdots & 0 \\
     0 &M^{\sigma} & \cdots & 0 \\
      \vdots &\vdots& \ddots & \vdots \\
      0 & 0&\cdots & M^{\sigma^{a-1}} \\
    \end{array}
  \right)
$$
where $M$ is the matrix of $\phi\circ E_f$ on $H_0(f)$.
It follows that
$${\rm Tr}_{{\mathbb Z}_q[\pi^{\frac{1}{d}}]}(\phi_f\mid H_0^{{\mathbb Z}/(a)})=0\text{ if } a>1.$$
Similarly, as
$$\phi_f^k=\sigma^{-k}\circ\phi^{k}\circ\prod_{j=0}^{k-1} E_f^{\sigma^j},$$
we can prove the following lemma.
\begin{lemma}If $q=p^a$, and
  $f\in{\mathbb F}_q[x_1,\cdots,x_n]$ has  smooth leading form, then
$${\rm Tr}_{{\mathbb Z}_q[\pi^{\frac{1}{d}}]}(\phi_f^k\mid H_0(f)^{{\mathbb Z}/(a)})=0\text{ whenever } k\not\equiv 0( {\rm mod } \ a).$$
\end{lemma}
\begin{theorem}[trace formula of Dwork operator]If
  $f\in{\mathbb F}_q[x_1,\cdots,x_n]$ has  smooth leading form, and $q=p^a$,  then
 $$L_{f,q}(t^a)^{(-1)^{n-1}}={\rm det}_{{\mathbb Z}_q[\pi^{\frac{1}{d}}]}(1-\phi_f t\mid H_0(f)^{{\mathbb Z}/(a)}).$$
\end{theorem}
\begin{proof}
As scalar extension does not change the characteristic polynomial of an operator, we have
 $${\rm det}_{{\mathbb Z}_p[\pi^{\frac{1}{d}}]}(1-\phi_f^at^a\mid H_0(f))={\rm det}_{{\mathbb Z}_q[\pi^{\frac{1}{d}}]}(1-\phi_f^a t^a\mid H_0(f)^{{\mathbb Z}/(a)}).$$
By the last lemma,
$${\rm det}_{{\mathbb Z}_q[\pi^{\frac{1}{d}}]}(1-\phi^a t^a\mid H_0(f)^{{\mathbb Z}/(a)})={\rm det}_{{\mathbb Z}_q[\pi^{\frac{1}{d}}]}(1-\phi_f t\mid H_0(f)^{{\mathbb Z}/(a)})^a.$$
By the trace formula of $a$-th power of Dwork operator,
 $$L_{f,q}(t^a)^{(-1)^{n-1}a}={\rm det}_{{\mathbb Z}_p[\pi^{\frac{1}{d}}]}(1-\phi_f^at^a\mid H_0(f)).$$
It follows that
$$
L_{f,q}(t^a)^{(-1)^{n-1}}={\rm det}_{{\mathbb Z}_q[\pi^{\frac{1}{d}}]}(1-\phi_f t\mid H_0(f)^{{\mathbb Z}/(a)}).$$
The theorem is proved.
\end{proof}
From the above corollary and the result of last subsection, we can deduce the following.
 The above theorem can be generalized to the exponential sums studied in [Adolphson and  Sperber 1987; 1989]. It can even be generalized to the $F$-crystals studied in [Katz 1979].
\subsection{Coefficients of $L$-function}
\paragraph{}
In this subsection we prove the following explicit formula for the coefficients of $L$-function of exponential sums.
\begin{theorem}[explicit formula]Assume that
  $f\in{\mathbb F}_q[x_1,\cdots,x_n]$ has  smooth leading form. Write
$$L_{f,q}(t)^{(-1)^{n-1}}=1+\sum_{k=1}^{(d-1)^n}(-1)^k\nu_k(f)t^k,$$
and
$$\phi\circ E_f(\pi^{\deg w}x^w)=\sum_{u\in\{1,2,\cdots,d-1\}^n}\pi^{\deg u-\deg w}c_{u,w}(f)\pi^{\deg u}x^u.$$
Then
$$\nu_k(f)=\sum_{A\in({\rm SF}(k))^a}\prod_{l\in{\mathbb Z}/(a)}
\det(c_{u,w}(f)^{\sigma^{l-1}})_{u\in A_l,w\in A_{l-1}},$$
where $q=p^a$.
\end{theorem}
\begin{proof}By the trace formula of Dwork operator,
$$\nu_k(f)=(-1)^{(a-1)k}\sum_{\stackrel{|S|=ak}{S\subseteq\{1,2,\cdots,d-1\}^n\times{\mathbb Z}/(a)}}\det(\delta_{i,l}c_{u,w}(f)^{\sigma^{l-1}})_{(u,i),(w,l)\in S}.$$
We claim that, if $S=\cup_{i\in{\mathbb Z}/(a)}A_i\times\{i\}$ with $|A_{i_0}|>k$, then
$$\det(\delta_{i,l}c_{u,w}(f)^{\sigma^{l-1}})_{(u,i),(w,l)\in S}
    =0.
$$
We may assume that $|A_{i_0}|>|A_{i_0+1}|$.
Then the columns
of the matrix $$(\delta_{i,l}c_{u,w}(f)^{\sigma^{l-1}})_{(u,i),(w,l)\in S}$$
indexed by $A_{i_0}\times\{i_0\}$ are linearly dependent as their coefficients are all zero except
those in the rows indexed by $A_{i_0+1}\times\{i_0+1\}$. It follows that
$$\det(\delta_{i,l}c_{u,w}(f)^{\sigma^{l-1}})_{(u,i),(w,l)\in S}=0.$$
The theorem now follows.\end{proof}
 \subsection{Lower Bound Estimation}
\paragraph{}
In this subsection prove the following theorem, which combined with the coincidence of premium polygon and Frobenius polygon, gives Theorem \ref{main-theorem-lower-bound-1}.
\begin{theorem}If $p>2d$, then
$${\rm GNP}(n,d;p)\geq{\rm PP}(n,d;p).$$\end{theorem}
\begin{proof}Let $f(x)\in\overline{{\mathbb F}}_p[x_1,\cdots,x_n]$ have smooth leading form. Write
$$E_f(x)=\sum_{u\in {\mathbb N}^n}\gamma_u(f)x^u.$$
It is easy to see that
$${\rm ord}_{\pi}(\gamma_u(f))\geq\lceil\deg u\rceil.$$
Write
$$\phi\circ E_f(\pi^{\deg w}x^w)=\sum_{u\in\{1,2,\cdots,d-1\}^n}\pi^{\deg u-\deg w}c_{u,w}(f)\pi^{\deg u}x^u.$$
Following Sperber [1986], we can prove that, if $u\in\{1,\cdots,d-1\}^n$, and
$v\in{\mathbb N}^n$ satisfies $\deg v<\deg u$,  then
$$\langle\pi^{\deg v}x^v,\pi^{\deg u}x^u\rangle=0\text{ on }  H_0(f).$$
It follows that
$$c_{u,w}(f)\equiv\sum_{\stackrel{v\in{\mathbb N}^n}{\deg v=\deg u}}\gamma_{pv-w}(f)\langle\pi^{\deg v}x^v,\pi^{\deg u}x^u\rangle ({\rm mod}\ \pi^{\lceil\deg(pu-w)\rceil+1}).$$
Hence
$${\rm ord}_{\pi}(c_{u,w}(f))\geq\lceil\deg(pu-w)\rceil.$$
Write
$$L_{f,q}(t)^{(-1)^{n-1}}=1+\sum_{k=1}^{(d-1)^n}(-1)^k\nu_k(f)t^k,$$
By the explicit formula for the coefficients of $L$-function,
$${\rm ord}_q(\nu_k(f))\geq{\rm Prem}(k),\ \forall k=0,1,\cdots,(d-1)^n.$$
Therefore
$${\rm NP}(f)\geq{\rm PP}(n,d;p).$$
Hence
$${\rm GNP}(n,d;p)\geq{\rm PP}(n,d;p).$$
The theorem is proved.\end{proof}
\subsection{The Exact Bound Estimation}
\paragraph{}
In this subsection  we prove the following theorem.
\begin{theorem}If $p>2d$, $k\leq{d+n-2\choose n}$ is a Frobenius vertex,  and
$$f(x)\in{\mathbb F}_{p^a}[x_1,\cdots,x_n]_{(d)}$$
has smooth leading form, then
$$\pi^{-a(p-1){\rm Prem}(k)}\nu_k(f)\equiv{\rm TH}^{(a)}(k)(f)({\rm mod}\ \pi).$$\end{theorem}
\begin{proof}We adopt the notations of last subsection.
Following Sperber [1986], we can prove that, if $u\in\{1,\cdots,d-1\}^n$, and
$v\in{\mathbb N}^n$ satisfies $\deg v<\deg u$,  then
$$\langle\pi^{\deg v}x^v,\pi^{\deg u}x^u\rangle=0\text{ on }  H_0(f).$$
It follows that, if $\deg u,\deg w\leq1$, then
$$c_{u,w}(f)\equiv\gamma_{pu-w}(f)({\rm mod}\ \pi^{\lceil\deg(pu-w)\rceil+1}).$$
It is easy to see that
 $$\pi^{-\lceil\deg(u)\rceil}\gamma_u(f)\equiv{\rm Poly}(u)(f)({\rm mod}\ \pi).$$
Thus, if $\deg u,\deg w\leq1$, then
 $$\pi^{-\lceil\deg(pu-w)\rceil}c_{u,w}(f)\equiv{\rm Poly}(pu-w)(f)({\rm mod}\ \pi).$$
If $k\leq{d\choose 2}$ is a Frobenius vertex, then by the explicit formula for the coefficients of $L$-function,
$$\pi^{-a(p-1){\rm Prem}(k)}\nu_k(f)\equiv{\rm TH}^{(a)}(k)(f)({\rm mod}\ \pi).$$
The theorem is proved.
\end{proof}
From the above theorem, we can infer the following corollary.
\begin{corollary}If $p>2d$, and for each Frobenius vertex $k\leq{d+n-2\choose n}$, there exist a positive integer $a$ and a polynomial
$$f(x)\in{\mathbb F}_{p^a}[x_1,\cdots,x_n]_{(d)}$$
whose leading form is smooth such that
$${\rm TH}^{(a)}(k)(f)\neq0,$$
 then
$${\rm GNP}(n,d;p)={\rm PP}(n,d;p)\text{ on }[0, {d+n-2\choose n}].$$\end{corollary}
\begin{corollary}If $n=2$, $p>2d$, and for each Frobenius vertex $k\leq{d\choose 2}$, there exist a positive integer $a$ and a polynomial
$$f(x)\in{\mathbb F}_{p^a}[x_1,\cdots,x_n]_{(d)}$$
whose leading form is smooth such that
$${\rm TH}^{(a)}(k)(f)\neq0,$$
 then
$${\rm GNP}(n,d;p)={\rm PP}(n,d;p).$$\end{corollary}
\begin{proof}By the last corollary, the present one follows from the symmetry of
the generic Newton polygon, and the symmetry of ${\rm PP}(n,d;p)$ which  comes from the symmetry of ${\rm FP}(n,d;p)$.
\end{proof}
\section{Polynomials in the open Stratum}
In this section we assume that $n=2$, and $d$ is even.
We prove that, if $p\equiv-1 ({\rm mod}\ d)$, and $p>d^3$, then for each Frobenius vertex $k\leq{d\choose 2}$, and for each $a>1$, there exists a polynomial
$$f(x)\in{\mathbb F}_{p^a}[x_1,x_2]_{(d)}$$
whose leading form is smooth such that
$${\rm TH}^{(a)}(k)(f)\neq0.$$
Theorem \ref{main-theorem-exact-bound} then follows.
\subsection{Specialization}
\paragraph{}
In this subsection, we define, for each Frobenius vertex $k$, a specialization ${\rm TH}_0^{(a)}(k)$ of
${\rm TH}^{(a)}(k)$ such that,  if $f(x)\in{\mathbb F}_{p^a}[x_1,x_2]$
has leading form
$$a_{d,0}x_1^d+a_{\frac{d}{2},\frac{d}{2}}x_1^{\frac{d}{2}}x_2^{\frac{d}{2}}+a_{0,d}x_2^d,$$
then
$${\rm TH}^{(a)}(k)(f)={\rm TH}_0^{(a)}(k)(f),\ a=1,2,\cdots.$$
\begin{definition}For $u\in{\mathbb N}^n$, we define
$${\rm Bund}_0(n,d)(u):=\{s\in{\rm Bund}(n,d)\mid s(w)=0,\ \forall w\in{\rm SV}\},$$
where
$${\rm SV}:=\{w\in{\rm Simp}(n,d)\cap{\mathbb Z}^n\mid \deg w<1, w=(0,d),(d,0), (\frac{d}{2},\frac{d}{2})\}$$\end{definition}
\begin{definition}For $u\in{\mathbb N}^n$, we define
$${\rm Poly}_0(u):=\sum_{s\in{\rm Bund}_0(n,d)(u)}{\rm Mono}(s).$$\end{definition}
\begin{definition}Assume that $n=2$ and $d$ is even. For an $a$-tuple of subsets of fundamental lattice points,
we define the set of specialization survivors of $A$ to be the set
  $${\rm Sym}_1^{(1)}(A):=\{\tau\in{\rm Sym}_0^{(1)}(A)\mid\prod_{l=1}^a\prod_{u\in A_l}{\rm Poly}_0(pu-\tau_l(u))\neq0\}.$$
\end{definition}
It is easy to see that, if $A$ is a subset of fundamental lattice points,
then
  $${\rm Sym}_1^{(1)}(A^a)=\left({\rm Sym}_1^{(1)}(A)\right)^a.$$
\begin{definition}For an $a$-tuple $A$
of subsets of fundamental lattice points, we define
$${\rm Poly}_0(A):=\sum_{\tau\in{\rm Sym}_1^{(a)}(A)}(-1)^{{\rm sgn}(\tau)}\prod_{l=1}^a\prod_{u\in A_l}{\rm Poly}_0(pu-\tau_l(u))^{p^{l-1}}.$$
\end{definition}
It is easy to see that, if $A$ is a subset of fundamental lattice points,
then
$${\rm Poly}_0(A^a)=\prod_{l=1}^a{\rm Poly}_0(A)^{p^{l-1}}.$$
\begin{definition}For Frobenius vertex $k$, we define
$${\rm TH}_0^{(a)}(k):=
\sum_{A\in{\rm SF}_0(k)^a}{\rm Poly}_0(A).$$
\end{definition}
It is easy to see that, if $k=\sum_{j=0}^ih_j$ is a Hodge vertex,
then
$${\rm TH}_0^{(a)}(k)=\prod_{l=1}^a{\rm Poly}_0(\cup_{j=0}^iW_j)^{p^{l-1}}.$$
\begin{lemma}Assume that $n=2$ and $d$ is even. Let $k$ be a Frobenius vertex. If $f(x)\in{\mathbb F}_{p^a}[x_1,x_2]$
has leading form
$$a_{d,0}x_1^d+a_{\frac{d}{2},\frac{d}{2}}x_1^{\frac{d}{2}}x_2^{\frac{d}{2}}+a_{0,d}x_2^d,$$
then
$${\rm TH}^{(a)}(k)(f)={\rm TH}_0^{(a)}(k)(f),\ a=1,2,\cdots.$$
\end{lemma}
\begin{proof}Obvious.\end{proof}
\subsection{Polarization}
\paragraph{}
In this subsection, we view the polynomials ${\rm TH}_0^{(a)}(k)$'s as polynomial over
$${\mathbb F}_p[a_{d,0},a_{\frac{d}{2},\frac{d}{2}},a_{0,d}],$$
and locate  their minimal forms.

\begin{definition}For each $j=0,1,\cdots,d$,
we define
a multiplicity function $m_j$ on ${\mathbb N}^n$ by the formulae
$$m_j(u):=\left\{
            \begin{array}{ll}
              -1, & \hbox{ if }\{\deg(u)\}=\frac{j}{d},\\
              0, & \hbox{ if }\{\deg(u)\}\neq\frac{j}{d},
            \end{array}
          \right.
\ j=1,\cdots, d-1,$$
$$m_d(u):=-\sum_{j=1}^{d-1}m_j(u),$$
and
$$m_0(u):=\left\{
            \begin{array}{ll}
              -1, & \hbox{ if } u_1=u_2,\\
              0, & \hbox{ if }u_1\neq u_2.
            \end{array}
          \right.
$$
\end{definition}
\begin{definition}Let
$A$ be an $a$-tuple of subsets of fundamental lattice points.  For each $j=0,1,\cdots,d$,
we define
a multiplicity function $m_j$ on ${\rm Sym}^{(1)}(A)$ by the formula
$$m_j(\tau):=\frac{p-1}{p^a-1}\sum_{l=1}^ap^{l-1}\sum_{u\in A_l}m_j(pu-\tau_l(u)),\ j=0,\cdots, d.$$
\end{definition}
\begin{definition}Let $(a_j)_{j=0}^d$ and $(b_j)_{j=0}^d$ be $d$-dimensional vectors. We define
$(a_j)_{j=0}^d<(b_j)_{j=0}^d$ if there is some $i$ such that
$$(a_i<b_i)\wedge(\forall j>i: a_j=b_j).$$
\end{definition}
\begin{definition}We define a system of multiplicity functions $\{m_j\}_{j=0}^d$ on the set of $a$-tuples
of subsets of fundamental lattice points by the formula
$$(m_j(A))_{j=0}^d:=\min_{\tau\in{\rm Sym}_1^{(1)}(A)}(m_j(\tau))_{j=0}^d.$$
\end{definition}
It is easy to see that, if
$A$ is a subset of fundamental lattice points,  then
$$m_j(A^a)=m_j(A),\ j=0,1,\cdots, d.$$
\begin{definition}For an $a$-tuple
$A$ of subsets of fundamental lattice points,  we define ${\rm Sym}_2^{(1)}(A)$ to be the subset of
${\rm Sym}_1^{(1)}(A)$ consisting of elements of minimal multiplicities, i.e.
$${\rm Sym}_2^{(1)}(A):=\{\tau\in{\rm Sym}_1^{(1)}(A)\mid m_j(\tau)=m_j(A),\ \forall j=0,1,\cdots,d\}.$$
\end{definition}
It is easy to see that, if
$A$ is a subset of fundamental lattice points,  then
$${\rm Sym}_2^{(1)}(A^a)=\left({\rm Sym}_2^{(1)}(A)\right)^a.$$
\begin{definition}For a Frobenius vertex $k$,
the multiplicities $m_j^{(a)}(k)$'s are defined by the formula
$$\left(m_j^{(a)}(k)\right)_{j=0}^d:=\min_{A\in{\rm SF}_0(k)^a}(\theta_j(A))_{j=0}^d.$$
\end{definition}
\begin{definition}For each Frobenius vertex $k$, we define
$${\rm SF}_1^{(a)}(k):=\{A\in{\rm SF}_0(k)^a\mid m_j(A)=m_j(k),\ \forall j=0,1,\cdots,d\}.$$
\end{definition}
It is easy to see that, if $k=\sum_{j=0}^ih_j$ is a Hodge vertex,  then
$${\rm SF}_1^{(a)}(k)=\{\left(\cup_{j=0}^iW_j\right)^a\}.$$
\begin{definition}For $u\in{\mathbb N}^n$, we define
$${\rm Bund}_1(n,d)(u):=\{s\in{\rm Bund}_0(n,d)\mid \deg_d(s)=\lceil\{\deg(u)\}\rceil\},$$
where
$$\deg_d(s)=\sum_{\deg w<1}s(w).$$\end{definition}
\begin{definition}For $u\in{\mathbb N}^n$, we define
$${\rm Poly}_1(u):=\sum_{s\in{\rm Bund}_1(n,d)(u)}{\rm Mono}(s).$$\end{definition}
\begin{definition}For an $a$-tuple
$A$ of subsets of fundamental lattice points,  we define
$${\rm Poly}_1(A):=\sum_{\tau\in{\rm Sym}_2^{(1)}(A)}(-1)^{{\rm sgn} \tau}\prod_{l=1}^a\prod_{u\in A_l}{\rm Poly}_1(pu-\tau_l(u))^{p^{l-1}}.$$
\end{definition}
It is easy to see that, if $A$ is a subset of fundamental lattice points,  then
$${\rm Poly}_1(A^a)=\prod_{l=1}^a{\rm Poly}_1(A)^{p^{l-1}}.$$
\begin{definition}For each Frobenius vertex $k$, we define
$${\rm TH}_1^{(a)}(k):=
\sum_{A\in{\rm SF}_1^{(a)}(k)}{\rm Poly}_1(A).$$
\end{definition}
It is easy to see that, if $k=\sum_{j=0}^ih_j$ is a Hodge vertex,  then
$${\rm TH}_1^{(a)}(k)=\prod_{l=1}^a{\rm Poly}_1(\cup_{j=0}^iW_j)^{p^{l-1}}.$$
\begin{lemma}There exists an partial order on the set
of monomials in
$${\mathbb F}_p[a_w\mid \deg w<1]$$
such that, for each Frobenius vertex $k$, and  for each positive integer $a$,
${\rm TH}_1^{(a)}(k)$ is the minimal form of
${\rm TH}_0^{(a)}(k)$.
\end{lemma}
\begin{proof}Obvious.\end{proof}
\subsection{Minimal Forms with Quadratic Frobenius}
\paragraph{}
In this subsection we assume that $p\equiv-1({\rm mod}\ d)$.
We locate the subsets of fundamental lattice points with minimal multiplicities, and
determine their twisted permutations of minimal multiplicities.
\begin{lemma}Assume that $p\equiv-1({\rm mod}\ d)$. Then
$$h_{i,1}=\left\{
            \begin{array}{ll}
              0, & \hbox{ if } i\leq\frac{d}{2},\\
              h_{d-i}, & \hbox{ if }\frac{d}{2}<i\leq d.
            \end{array}
          \right.
$$
\end{lemma}
\begin{proof}Obvious. \end{proof}
\begin{definition}Assume that $n=2$, $d$ is even, and $p\equiv-1({\rm mod}\ d)$. For $d-j\leq j\leq d$,
we define
$$W_{j,1}:=\{u\in(0,d)^n\cap{\mathbb Z}^n\mid \deg(u)=\frac{j}{d}, u_1<\frac{d}{2}, u_2<\frac{d}{2}\},$$
and $$W_{j,0}:=W_j\setminus W_{j,1}.$$\end{definition}
\begin{lemma}If $d-i<i\leq d$, then
$${\rm Sym}_2^{(1)}(W_{i,1}\cup\cup_{j=0}^{i-1}W_j)=\{\tau_0\},$$
where $\tau_0$ is defined by the formula
$$\tau_0(u)=\left\{
              \begin{array}{ll}
                (\frac{d}{2},\frac{d}{2})-u, & \hbox{ if }  u\in W_j,\  j\leq d-j\leq i,\\
                (\frac{d}{2},\frac{d}{2})-u, & \hbox{ if } u\in W_{j,1}, d-j\leq j\leq i,\\
                (u_2,u_1), & \hbox{ if } u\in W_{j,0},  d-j\leq j<i,\\
                (u_2,u_1), & \hbox{ if }u\in W_j, j\leq i<d-j
              \end{array}
            \right.
$$
\end{lemma}
\begin{proof}It is easy to see that
$$\tau_0\in{\rm Sym}_1^{(1)}(W_{i,1}\cup\cup_{j=0}^{i-1}W_j).$$
Let
$$\tau\in{\rm Sym}_1^{(1)}(W_{i,1}\cup\cup_{j=0}^{i-1}W_j).$$
be a permutation such that
$$(m_l(\tau))_{l=0}^d\leq (m_l(\tau_0))_{l=0}^d.$$
Firstly, we claim that
$$m_d(\tau)=m_d(\tau_0),$$
$$\tau(W_j)\left\{
             \begin{array}{ll}
               \subseteq W_{d-j}, & \hbox{ if } j\leq d-j\leq i,\\
               \supseteq W_{d-j}, & \hbox{ if } d-j<j\leq i,
             \end{array}
           \right.
$$
and
$$\tau(W_{i,1})=W_{d-i}.$$
Indeed
\begin{eqnarray*}
  m_d(\tau) &=&\sum_{u\in W_{i,1}}m_d(pu-\tau(u))+ \sum_{j=0}^{i-1}\sum_{u\in W_j}m_d(pu-\tau(u)) \\
   &\geq&  \sum_{j\leq i<d-j}|W_j|+\sum_{d-j< j< i}\sum_{\stackrel{u\in W_j}{\tau(u)\not\in W_{d-j}}} 1\\ &&+\sum_{\stackrel{u\in W_{i,1}}{\tau(u)\not\in W_{d-i}}} 1+\sum_{j\leq d-j\leq i}\sum_{\stackrel{u\in W_j}{\tau(u)\not\in W_{d-j}}} 1\\
   &\geq& m_d(\tau_0)+\sum_{d-j< j< i}(\sum_{\stackrel{u\in W_j}{\tau(u)\not\in W_{d-j}}} 1-(|W_j|-|W_{d-j}|))\\
&&+\sum_{\stackrel{u\in W_{i,1}}{\tau(u)\not\in W_{d-i}}} 1+\sum_{j\leq d-j\leq i}\sum_{\stackrel{u\in W_j}{\tau(u)\not\in W_{d-j}}} 1\\
   &\geq& m_d(\tau_0)+\sum_{d-j< j< i}(|W_{d-j}|-\sum_{\stackrel{u\in W_j}{\tau(u)\in W_{d-j}}} 1)\\
&&+\sum_{\stackrel{u\in W_{i,1}}{\tau(u)\not\in W_{d-i}}} 1+\sum_{j\leq  d-j\leq i}\sum_{\stackrel{u\in W_j}{\tau(u)\not\in W_{d-j}}} 1\\
&\geq &m_d(\tau_0),
\end{eqnarray*}
where equality holds if and only if$$\tau(W_j)\left\{
             \begin{array}{ll}
               \subseteq W_{d-j}, & \hbox{ if } j\leq d-j< i,\\
               \supseteq W_{d-j}, & \hbox{ if } d-j<j\leq i,
             \end{array}
           \right.
$$
and
$$\tau(W_{i,1})=W_{d-i}.$$
Secondly, we claim that, if $l=1,2,\cdots,d-1$, then
$$m_l(\tau)= m_l(\tau_0),$$
and
$$\tau(W_j)\subseteq \left\{
                       \begin{array}{ll}
                         W_j\cup W_{d-j}, & \hbox{ if } d-j\leq j< i,\ l=2d-2j,\\
                         W_j, & \hbox{ if }j\leq i <d-j,\ l=d-2j
                       \end{array}
                     \right.
$$
In fact, as $\tau(W_j)\subseteq W_{d-j}$  when $j\leq d-j\leq i$, we have
$$\sum_{j\leq d-j\leq i}\sum_{u\in W_j}m_l(pu-\tau(u))=0.$$
Similarly,
$$\sum_{u\in W_{i,1}}m_l(pu-\tau(u))=0.$$
So
\begin{eqnarray*}
  m_l(\tau) &=& \sum_{j=0}^i\sum_{u\in W_j}m_l(pu-\tau(u)) \\
   &= & \sum_{d-j<j< i}\sum_{u\in W_j}m_l(pu-\tau(u))\\
   & & +\sum_{j\leq i< d-j}\sum_{u\in W_j}m_l(pu-\tau(u)).
\end{eqnarray*}
By induction on $l$,
$$\tau(W_j)\subseteq \left\{
                       \begin{array}{ll}
                         W_j\cup W_{d-j}, & \hbox{ if } d-j\leq j< i,\ 2d-2j>l,\\
                         W_j, & \hbox{ if }j\leq i <d-j,\ d-2j>l
                       \end{array}
                     \right.
$$
So
$$\sum_{\stackrel{d-j<j\leq i}{2d-2j>l}}\sum_{u\in W_j}m_l(pu-\tau(u))+\sum_{\stackrel{j\leq i< d-j}{d-2j>l}}\sum_{u\in W_j}m_l(pu-\tau(u))=0,$$
and
$$W_j\subseteq \left\{
                       \begin{array}{ll}
                         \tau(W_j)\cup \tau(W_{d-j}), & \hbox{ if } d-j\leq j< i,\ 2d-2j>l,\\
                         \tau(W_j), & \hbox{ if }j\leq i <d-j,\ d-2j>l.
                       \end{array}
                     \right.
$$
It follows that
\begin{eqnarray*}
  m_l(\tau)& = &\sum_{\stackrel{d-j<j\leq i}{2d-2j\leq l}}\sum_{\stackrel{u\in W_j}{\deg\tau(u)\geq \frac{2d-l}{2d}}}m_l(pu-\tau(u))\\
& &+\sum_{\stackrel{j\leq i <d-j}{d-2j\leq l}}\sum_{\stackrel{u\in W_j}{\deg\tau(u)\geq \frac{d-l}{2d}}}m_l(pu-\tau(u))\\
&= &\sum_{\stackrel{d-j<j\leq i}{2d-2j=l}}\sum_{\stackrel{u\in W_j}{\deg\tau(u)= \frac{2d-l}{2d}}}m_l(pu-\tau(u))\\
& &+\sum_{\stackrel{j\leq i <d-j}{d-2j=l}}\sum_{\stackrel{u\in W_j}{\deg\tau(u)= \frac{d-l}{2d}}}m_l(pu-\tau(u))\\
&\geq&m_l(\tau_0),
\end{eqnarray*}
where equality holds if and only if

$$\tau(W_j)\subseteq \left\{
                       \begin{array}{ll}
                         W_j\cup W_{d-j}, & \hbox{ if } d-j\leq j\leq i,\ l=2d-2j,\\
                         W_j, & \hbox{ if }j\leq i <d-j,\ l=d-2j
                       \end{array}
                     \right.
$$
Thirdly, we claim that
$$ \left\{
                       \begin{array}{ll}
                        \tau(u)=(\frac{d}{2},\frac{d}{2})-u, & \hbox{ if }  u\in W_j,\  j\leq d-j\leq i,\\
                        \tau(u)=(\frac{d}{2},\frac{d}{2})-u, & \hbox{ if }  u\in W_{j,1},\  d-j\leq j\leq i,\\
                        (d,d)-u-\tau(u)\in{\mathbb N}^2, & \hbox{ if }  u\in W_{j,0},\  d-j\leq j< i,\\
                          (\frac{d}{2},\frac{d}{2})-u-\tau(u)\in{\mathbb N}^2, & \hbox{ if }j\leq i <d-j.
                       \end{array}
                     \right.
$$
Indeed,
$${\rm Bund}_1(pu-\tau(u))\neq\emptyset,\ \forall u\in W_{i,1}\cup\cup_{j=0}^{i-1}.$$
For simplicity, we assume that $u\in W_j$ with $j\leq d- j\leq i$. The other cases can be proved similarly.
Then
$$(d,d)-u-\tau(u) \equiv0 \left({\rm mod}\ {\mathbb Z}(d,0)+{\mathbb Z}(0,d)\right),$$
or
$$(d,d)-u-\tau(u) \equiv (\frac{d}{2},\frac{d}{2})\left({\rm mod}\ {\mathbb Z}(d,0)+{\mathbb Z}(0,d)\right).$$
As the former congruence is impossible, while the latter implies an equality, we conclude that
$$\tau(u)=(\frac{d}{2},\frac{d}{2})-u.$$
\paragraph{}
Finally, as $m_0(\tau)\leq m_0(\tau_0)$, we conclude that
$\tau=\tau_0$.
The lemma is proved.
\end{proof}
Similarly, we can prove the following lemma.
\begin{lemma}If $i\leq d$, then
$${\rm Sym}_2^{(1)}(\cup_{j=0}^iW_j)=\{\tau_0\},$$
where $\tau_0$ is defined by the formula
$$\tau_0(u)=\left\{
              \begin{array}{ll}
                (\frac{d}{2},\frac{d}{2})-u, & \hbox{ if }  u\in W_j,\  j\leq d-j\leq i,\\
                (\frac{d}{2},\frac{d}{2})-u, & \hbox{ if } u\in W_{j,1}, d-j\leq j\leq i,\\
                (u_2,u_1), & \hbox{ if } u\in W_{j,0},  d-j\leq j\leq i,\\
                (u_2,u_1), & \hbox{ if }u\in W_j, j\leq i<d-j
              \end{array}
            \right.
$$
\end{lemma}
We now prove the following lemma.
\begin{lemma}Assume that $n=2$, $d$ is even, and $p\equiv-1({\rm mod}\ d)$.
If $k=h_{i,1}+\sum_{j=0}^{i-1}h_j$, then
$${\rm SF}_1^{(a)}(k)=\{(W_{i,1}\cup\cup_{j=0}^{i-1}W_j)^a\}.$$\end{lemma}
\begin{proof}
Let $A\in {\rm SF}_1^{(a)}(k)$, and $\tau\in{\rm Sym}_2^{(1)}(A)$.
Then
$$(m_l(\tau))_{l=0}^d\leq (m_l(\tau_0))_{l=0}^d.$$
Hence
$$\tau_j(W_{d-i})=\left\{u\in A_{j-1}\mid \deg u=\frac{i}{d}\right\},\ \forall j=0,1,\cdots,a-1.$$
As
$${\rm Bund}_1(pu-\tau_j(u))\neq\emptyset,\ \forall u\in W_{i,1},\ \forall j=0,1,\cdots,a-1,$$
$$(d,d)-u-\tau_j(u) \equiv0 \left({\rm mod}\ {\mathbb Z}(d,0)+{\mathbb Z}(0,d)\right),$$
or
$$(d,d)-u-\tau_j(u) \equiv (\frac{d}{2},\frac{d}{2})\left({\rm mod}\ {\mathbb Z}(d,0)+{\mathbb Z}(0,d)\right).$$
As the former congruence is impossible, while the latter implies an equality, we conclude that
$$\tau_j(u)=(\frac{d}{2},\frac{d}{2})-u.$$
This implies that
$$\left\{u\in A_{j-1}\mid \deg u=\frac{i}{d}\right\}\subseteq W_{i,1},\ \forall j=0,1,\cdots,a-1.$$
Counting their cardinality, we conclude that
$$\left\{u\in A_{j-1}\mid \deg u=\frac{i}{d}\right\}= W_{i,1},\ \forall j=0,1,\cdots,a-1.$$
That is,
$$A=(W_{i,1}\cup\cup_{j=0}^{i-1}W_j)^a.$$
Therefore
$${\rm SF}_1^{(a)}(k)=\{(W_{i,1}\cup\cup_{j=0}^{i-1}W_j)^a\}.$$
The lemma is proved.
\end{proof}
\subsection{Nonvanishing of Twisted Hasse Polynomial}
\paragraph{}
 In this subsection, we prove the following theorem.
\begin{theorem}Assume that $n=2$, $d$ is even, and $p\equiv-1 ({\rm mod}\ d)$. If $p>d^3$ and $a\geq2$, then
there exists
$$f(x)\in{\mathbb F}_{p^a}[x_1,x_2]_{(d)}$$
whose leading form is
$$a_{d,0}x_1^d+a_{\frac{d}{2},\frac{d}{2}}x_1^{\frac{d}{2}}x_2^{\frac{d}{2}}+a_{0,d}x_2^d$$
such that
$$a_{d,0}a_{0,d}(a_{\frac{d}{2},\frac{d}{2}}^2-4a_{d,0}a_{0,d})\neq0,$$ and, for each Frobenius vertex $k$,
$${\rm TH}_0^{(a)}(k)(f)\neq0.$$\end{theorem}
It is easy to see that the above theorem follows from the following lemma.
\begin{lemma}Assume that $n=2$, $d$ is even, and $p\equiv-1 ({\rm mod}\ d)$. If $k\leq {d\choose 2}$ is a Frobenius vertex, then there exist a nonzero monomial
$${\rm Int}(k)\in {\mathbb F}_p[a_w\mid \deg w<1],$$
and a nonzero polynomial
$${\rm Fac}(k)\in{\mathbb F}_p[a_{d,0},a_{\frac{d}{2},\frac{d}{2}},a_{0,d}]$$
of degree less than $(d-1)^2p$,
such that
$${\rm TH}_1^{(a)}(k)=\prod_{l=1}^a\left({\rm Fac}(k){\rm Int}(k)\right)^{p^{l-1}},\ a=1,2,\cdots.$$\end{lemma}
\begin{proof}
For simplicity, we assume that
$k=\sum_{j=0}^nh_j$. The other case can be proved similarly.
Then\begin{eqnarray*}
      {\rm TH}_1^{(1)}(k) &=& \prod_{j\leq i<d-j}\prod_{u\in W_j}{\rm Poly}_1(pu-(u_2,u_1)) \\
       &\times& \prod_{j\leq d-j\leq i}\prod_{u\in W_j}{\rm Poly}_1((p+1)u-(\frac{d}{2},\frac{d}{2})) \\
       &\times& \prod_{d-j\leq j\leq i}\prod_{u\in W_{j,1}}{\rm Poly}_1((p+1)u-(\frac{d}{2},\frac{d}{2})) \\
&\times&\prod_{d-j<j\leq i}\prod_{u\in W_{j,0}}{\rm Poly}_1(pu-(u_2,u_1)).
    \end{eqnarray*}
Notice that, if $u\in W_j$ with $j\leq i<d-j$, then
$${\rm Poly}_1(pu-(u_2,u_1))={\rm Poly}_1(\frac{d-2j}{2}(1,1)){\rm Poly}_1((p+1)u-(\frac{d}{2},\frac{d}{2})).$$
Similarly,
if $u\in W_{j,0}$ with $d-j< j\leq i$, then
$${\rm Poly}_1(pu-(u_2,u_1))={\rm Poly}_1((d-j,d-j)){\rm Poly}_1((p+1)u-(d,d)).$$
Let
\begin{eqnarray*}
      {\rm Int}(k) &=& \prod_{j\leq i<d-j}\prod_{u\in W_j}{\rm Poly}_1(\frac{d-2j}{2}(1,1)) \\
&\times&\prod_{d-j<j\leq i}\prod_{u\in W_{j,0}}{\rm Poly}_1((d-j,d-j)).
    \end{eqnarray*}
Then $${\rm Int}(k)\in{\mathbb F}_p[a_w\mid \deg w<1]$$
is a nonzero monomial.
Let\begin{eqnarray*}
      {\rm Fac}(k) &=& \prod_{j\leq i<d-j}\prod_{u\in W_j}{\rm Poly}_1((p+1)u-(\frac{d}{2},\frac{d}{2})) \\
       &\times& \prod_{j\leq d-j\leq i}\prod_{u\in W_j}{\rm Poly}_1((p+1)u-(\frac{d}{2},\frac{d}{2})) \\
       &\times& \prod_{d-j\leq j\leq i}\prod_{u\in W_{j,1}}{\rm Poly}_1((p+1)u-(\frac{d}{2},\frac{d}{2})) \\
&\times&\prod_{d-j<j\leq i}\prod_{u\in W_{j,0}}{\rm Poly}_1((p+1)u-(d,d)).
    \end{eqnarray*}
Then
$${\rm Fac}(k)\in{\mathbb F}_p[a_{d,0},a_{\frac{d}{2},\frac{d}{2}},a_{0,d}],$$
and
$${\rm TH}_1^{(a)}=\prod_{l=1}^a({\rm Fac}(k){\rm Int}(k))^{p^{l-1}}.$$
Each factor of ${\rm Fac}(k)$, being a nonzero combination of monomials of different exponents, is nonzero. The degree of ${\rm Fac}(k)$ is easily seen to be bounded by
$(d-1)^2p$.
 \end{proof}
%=========================================================

\section{Table of Nations}
In this section we list the notions of present paper.
\begin{itemize}
\item ${\rm Bund}(n,d)$  is the integral bundle on ${\rm Simp}(n,d)$.
\item ${\rm Bund}(n,d)(u)$ is the set of sections of minimal degree which represents $u$.
\item ${\rm Bund}_0(n,d)(u)$ is the subset of ${\rm Bund}(n,d)(u)$ consisting of sections which vanishes at specialized variables.
\item ${\rm Bund}_1(n,d)(u)$ is the subset of ${\rm Bund}_0(n,d)(u)$ consisting of sections of minimal interior degree.
\item $\deg(A)$ is the degree of  a subset of fundamental lattice points.
\item$\deg (s)$ is the degree of integral section $s$.
\item $\deg(u)$ is the degree of lattice point $u$.
\item ${\rm Disc}(n,d)$ is the discriminant on the set of homogenous polynomials in $\overline{{\mathbb F}}_p[x_1,\cdots, x_n]_{(d)}$.
\item $E(x)$ is Artin-Hasse exponential.
\item $E_f(x)$ is a power series in $n$-variables building from $E$ and $f$.
\item $\widehat{E}_f(x)$ is an infinite product constructed from $E_f(x)$.
\item $\overline{{\mathbb F}}_p[x_1,\cdots,x_n]_{(d)}$ is the set of degree-$d$ polynomials in $\overline{{\mathbb F}}_p[x_1,\cdots,x_n]$.
\item ${\rm Fac}(k)$ is the facial factor of ${\rm TH}_1^{(1)}(k)$.
\item ${\rm FP}(n,d;p)$ is Frobenius polygon of ${\rm Simp}(n,d)$.
\item ${\rm FP}^{(i)}(n,d;p)$'s are fitted Frobenius polygons.
\item $H_0(f)$ is the $p$-adic homology space of $f$.
\item $h_j$'s are Hodge numbers of ${\rm Simp}(n,d)$.
\item $h_{j,\epsilon}$'s are Frobenius numbers of ${\rm Simp}(n,d)$.
\item $H_j$'s are arithmetic  Hodge sums of ${\rm Simp}(n,d)$.
\item ${\rm Hass}(n,d;p)$ is Hasse polynomial on $\overline{{\mathbb F}}_p[x_1,x_2,\cdots, x_n]_{(d)}$.
  \item ${\rm HP}(n,d)$ is Hodge polygon of ${\rm Simp}(n,d)$.
\item ${\rm Int}(k)$ is the interior factor of ${\rm TH}_1^{(1)}(k)$.
\item ${\rm Iso}(A,B)$ is the set of isomorphisms from $A$ to $B$.
\item ${\rm Iso}_{\deg}(A,B)$ is the set of degree-preserving isomorphisms from $A$ to $B$.
\item $L_{f,q}(t)$ is the $L$-function associated to $(f,q)$.
\item $m(\alpha^{-1})$ is the order of $\alpha^{-1}$ as a zero of $L_{f,q}(t)^{(-1)^{n-1}}$.
\item $m_j(u)$'s are multiplicities of a lattice point.
\item $m_j(\tau)$'s are multiplicities of a twisted permutation.
\item $m_j(A)$'s are multiplicities of a tuple of subsets of fundamental lattice points.
\item $m_j(k)$'s are multiplicities at an integer $k$.
\item ${\rm Mono}(s)$ is the monomial representation of a section involving Artin-Hasse coefficients.
\item ${\rm NP}(f)$ is Newton polygon of $f$.
\item ${\rm GNP}(n,d;p)$ is the generic Newton polygon of $\overline{{\mathbb F}}_p[x_1,\cdots,x_n]_{(d)}$.
\item ${\rm Poly}(A)$ is the polynomial of a tuple of subsets of fundamental lattice points.
\item ${\rm Poly}_0(A)$ is the specialization of ${\rm Poly}(A)$.
\item ${\rm Poly}_1(A)$ is the minimal form of ${\rm Poly}_0(A)$.
\item ${\rm Poly}(u)$ is the polynomial of a lattice point.
\item ${\rm Poly}_0(u)$ is the specialization of ${\rm Poly}(u)$.
\item ${\rm Poly}_1(u)$ is the minimal form of ${\rm Poly}_0(u)$.
\item ${\rm PP}(n,d;p)$ is the premium polygon of ${\rm Simp}(n,d)$.
\item${\rm Prem}(A)$ is Frobenius premium of a tuple of subsets of fundamental lattice points.
\item${\rm Prem}(k)$ is Frobenius premium at an integer.
\item${\rm Prem}(\tau)$ is Frobenius premium of a twisted permutation.
\item $S_{f,q}(k)$'s are exponential sums associated to $(f,q)$.
\item ${\rm SF}(k)$ is the set of subsets of $(0,d)^n\cap{\mathbb Z}^n$ of cardinality $k$.
\item ${\rm SF}_0(k)$ is the subset of ${\rm SF}(k)$ consisting of elements of minimal degree.
\item ${\rm SF}_1^{(a)}(k)$ is the subset of ${\rm SF}_0(k)^a$ consisting of elements of minimal multiplicities.
  \item ${\rm Simp}(n,d)$ is the $d$-multiple of the standard unit simplex in ${\mathbb R}^n$.
\item ${\rm Sym}^{(1)}(A)$ is the set of twisted permutation of $A$.
\item ${\rm Sym}_0^{(1)}(A)$ is the subset of ${\rm Sym}^{(1)}(A)$ consisting of elements of minimal premium.
\item ${\rm Sym}_1^{(1)}(A)$ is the subset of ${\rm Sym}_0^{(1)}(A)$ consisting of elements which makes a contribution after specialization.
\item ${\rm Sym}_2^{(1)}(A)$ is the subset of ${\rm Sym}_1^{(1)}(A)$ consisting of elements of minimal multiplicities.
\item ${\rm TH}^{(a)}(k)$ is the twisted Hasse polynomial at $k$.
\item ${\rm TH}_0^{(a)}(k)$ is the specialization of ${\rm TH}^{(a)}(k)$.
\item ${\rm TH}_1^{(a)}(k)$ is the minimal form of ${\rm TH}_1^{(a)}(k)$.
\item ${\rm vec}(s)$ the lattice point represented by $s$.
\item $W_j$ is the set of fundamental lattice points of degree $\frac{j}{d}$.
\item $W_{j,1}$ is the set of median fundamental lattice points of degree $\frac{j}{d}$.
\item $W_{j,0}$ is the set of marginal fundamental lattice points of degree $\frac{j}{d}$.
\item $\delta_{ij}$'s are Knonecker's $\delta$-symbols.
\item $\lambda_j$'s are coefficients of Artin-Hasse exponential.
\item $\overline{\lambda}_j$'s are characteristic-$p$ coefficients of Artin-Hasse exponential.
\item $\partial$ is the boundary operator obtaining $H_0(f)$ aside from being partial derivative.
\item $\phi$ is Frobenius action on power series.
\item $\phi_f$ is Dwork operator of $f$.
\item $\pi$ is the Artin-Hasse uniformizer of $\zeta_p-1$.
\item $\sigma_j$'s are conjugates of Frobenius  on ${\mathbb Z}/(d)$.
\item $\sigma$ is Frobenius element of ${\rm Gal}(\overline{{\mathbb F}}_p/{\mathbb F}_p)$.
\item $\varpi_{j,\epsilon}$'s are Frobenius slopes of ${\rm Simp}(n,d)$.
\item $\varpi_j^{(i)}$'s are fitted Frobenius slopes.
\item $\zeta_p$ is a $p$-th root of unity.
\item $\lceil\cdot\rceil$ is the least integer no less than a rational number.
\item $[\cdot]$ is the greatest integer no more than a rational number aside from being interval symbol.
\item $\{\cdot\}$ is the fractional part of a rational number aside from being set symbol.
\end{itemize}

\end{document}